\documentclass[a4paper,12pt]{article}

\usepackage[left=2cm,right=2cm, top=2cm,bottom=3cm,bindingoffset=0cm]{geometry}

\usepackage{verbatim}
\usepackage{amsmath}
\usepackage{amsthm}
\usepackage{amssymb}
\usepackage{delarray}
\usepackage{cite}
\usepackage{hyperref}
\usepackage{mathrsfs}
\usepackage{tikz}
\usetikzlibrary{patterns}
\usepackage{caption}
\DeclareCaptionLabelSeparator{dot}{. }
\newcommand{\al}{\alpha}
\newcommand{\be}{\beta}
\newcommand{\ga}{\gamma}
\newcommand{\de}{\delta}
\newcommand{\la}{\lambda}
\newcommand{\om}{\omega}

\newcommand{\eps}{\varepsilon}
\newcommand{\vv}{\varphi}

\theoremstyle{plain}

\numberwithin{equation}{section}

\newtheorem{thm}{Theorem}[section]
\newtheorem{lem}[thm]{Lemma}
\newtheorem{prop}[thm]{Proposition}
\newtheorem{cor}[thm]{Corollary}

\theoremstyle{definition}

\newtheorem{example}[thm]{Example}

\newtheorem{ip}[thm]{Inverse Problem}

\theoremstyle{remark}

\newtheorem{remark}[thm]{Remark}

\DeclareMathOperator*{\Res}{Res}

 \allowdisplaybreaks

\begin{document}

\begin{center}
{\Large\bf Uniform stability of the inverse problem \\[0.2cm] for the non-self-adjoint Sturm-Liouville operator}
\\[0.5cm]
{\bf Natalia P. Bondarenko}
\end{center}

\vspace{0.5cm}

{\bf Abstract.} In this paper, we develop a new approach to investigation of the uniform stability for inverse spectral problems. We consider the non-self-adjoint Sturm-Liouville problem that consists in the recovery of the potential and the parameters of the boundary conditions from the eigenvalues and the generalized weight numbers. The special case of simple eigenvalues, as well as the general case with multiple eigenvalues are studied. We find various subsets in the space of spectral data, on which the inverse mapping is Lipschitz continuous, and obtain the corresponding unconditional uniform stability estimates. Furthermore, the conditional uniform stability of the inverse problem under a priori restrictions on the potential is studied. In addition, we prove the uniform stability of the inverse problem by the Cauchy data, which are convenient for numerical reconstruction of the potential and for applications to partial inverse problems.

\medskip


\medskip

  
\vspace{1cm}

\section{Introduction} \label{sec:intr}

This paper deals with the following non-self-adjoint Sturm-Liouville problem $L = L(q, h, H)$:
\begin{gather} \label{eqv}
-y'' + q(x) y = \la y, \quad x \in (0, \pi), \\ \label{bc}
y'(0) - h y(0) = 0, \quad y'(\pi) + H y(\pi) = 0,
\end{gather}
where $q \in L_2(0,\pi)$ is a complex-valued function, $h$ and $H$ are complex constants, and $\la$ is the spectral parameter. We consider an inverse spectral problem that consists in the recovery of the parameters $q$, $h$, and $H$ from the eigenvalues $\{ \la_n \}_{n \ge 1}$ and the generalized weight numbers $\{ \al_n \}_{n \ge 1}$. This paper is focused on the study of the uniform stability for this inverse problem. We study the case of simple eigenvalues, as well as the general case when a finite number of multiple eigenvalues are possible.
Our goal is to develop a general approach to investigating the uniform stability of inverse spectral problems for various classes of operators, including non-self-adjoint ones with multiple eigenvalues. 

Inverse Sturm-Liouville problems and their generalizations have been studied for more than seventy years (see the monographs \cite{Lev84, PT87, FY01, Mar11, Krav20} and references therein). Such problems arise in quantum and classical mechanics, geophysics, acoustics, nanotechnology, and other applications. The major part of the research on inverse spectral theory is concerned with self-adjoint operators. Anyway, several approaches to the non-self-adjoint Sturm-Liouville problems have also been developed, see \cite{FY01} for the case of simple eigenvalues and \cite{Tkach02, But07, AHM08} for the case of multiplicities. Nevertheless, the inverse Sturm-Liouville theory still contains open questions, some of which are addressed in this paper.

In the fundamental study \cite{Borg46}, Borg has proved the uniqueness
and the local solvability of the inverse problem, which consists in determining the Sturm-Liouville potential from two spectra. Gelfand and Levitan \cite{GL51} developed a method for reconstruction of the potential from the spectral data (spectral function), which are equivalent to Borg's two spectra. Basing on the Gelfand-Levitan method, the following necessary and sufficient conditions on the spectral data $\{ \la_n, \al_n \}_{n \ge 1}$ of the \textit{self-adjoint} problem \eqref{eqv}--\eqref{bc} (i.e. for real-valued $q$, $h$, and $H$) have been obtained (see, e.g., \cite[Theorem 1.3.2]{FY01}):
\begin{gather} \label{sdself}
\la_n \in \mathbb R, \quad \la_n \ne \la_k \: (n \ne k), \quad \al_n > 0, \quad n \ge 1, \\
\label{asympt}
\rho_n := \sqrt{\la_n} = n - 1 + \frac{\om}{\pi n} + \frac{\varkappa_n}{n}, \quad
\al_n = \frac{2}{\pi} + \frac{s_n}{n}, \quad \{ \varkappa_n \}, \, \{ s_n \} \in l_2,
\end{gather}
where
\begin{equation} \label{defom}
\om := h + H + \frac{1}{2} \int_0^{\pi} q(x) \, dx.
\end{equation}

The stability of inverse Sturm-Liouville problems has been investigated in \cite{FY01, Hoch77, McL88, BSY13, BK19, Bond20, BB17} and many other studies. However, the stability in the mentioned papers has a \textit{local} nature. In other words, those studies dealt with local perturbations of spectral data for a fixed operator, and constants in stability estimates depended on that operator. Such estimates are inconvenient for practical applications, since in practice an exact solution is typically unknown. Therefore, \textit{uniform stability} estimates are needed. Note that the inverse Sturm-Liouville problem cannot be uniformly stable for all the collections of the spectral data satisfying \eqref{sdself} and \eqref{asympt}. Indeed, if $\la_n$ tends to $\la_{n+1}$ or $\al_n$ tends to zero, then the norm of the potential $q$ can grow infinitely. Thus, it is important to find such sets of spectral data, on which the uniform stability holds. We distinguish \textit{unconditional} uniform stability, for which some restrictions are imposed directly on the spectral data, and \textit{conditional} uniform stability, for which the potential is required to fulfill some a priori bounds.

Unconditional uniform stability for the first time was obtained by Marchenko and Ostrovskii \cite{MO80} for the periodic inverse problem. In 2010, Savchuk and Shkalikov \cite{SS10} have proved the uniform stability for the recovery of the self-adjoint Sturm-Liouville operators with distribution potentials of the Sobolev spaces $W_2^{\alpha}$, $\al > -1$, from the two spectra and from the spectral data. Note that the results of \cite{SS10} appeared to be new even for regular potentials $q \in L_2$. Later on, the uniform stability of inverse spectral problems for Sturm-Liouville and Dirac operators has been obtained by Hryniv \cite{Hryn11-1, Hryn11-2} by using a different method.

Uniform stability of inverse spectral problems is closely related to the recovery of operators from finite sets spectral data. The results of Savchuk and Shkalikov \cite{SS10} allowed them to obtain estimates for approximation of a potential by finite data, which are convenient for practical use (see \cite{SS14, Sav16}). Reconstruction of the Sturm-Liouville operators on finite intervals by finite data was previously investigated by Ryabushko \cite{Ryab73}, Alekseev \cite{Al86}, Marletta and Weikard \cite{MW05}. Stability of inverse problems by finite data on the line in different statements was considered by Aktosun \cite{Akt87} and Hitrik \cite{Hit00}. A great contribution to the study of inverse Sturm-Liouville problems on the half-line by a finite part of the scattering data and of the spectral function has been made by Marchenko and his co-authors. In their works \cite{Mar68, LM69, MM70, Mar11}, the difference of two potentials was estimated under the assumption that a finite part of the corresponding spectral data are equal and under some a priori restrictions on the potentials. In recent years, those restrictions were weakened by Xu \cite{Xu21}. Moreover, the ideas of Marchenko were generalized to the matrix Sturm-Liouville operators with possibly multiple eigenvalues in \cite{XB22}. We also mention that stability of recovering the potential from a finite number of resonances was investigated by Marletta et al \cite{MSW09} for the half-line and by Bledsoe \cite{Bled12} for the line.

In the recent paper \cite{GMXA23}, uniform stability estimates have been obtained for the inverse Sturm-Liouville problem by two spectra in the non-self-adjoint case under a priori restriction of the potential norm. Furthermore, conditional uniform stability for the inverse scattering problem on the lasso-graph was investigated by Mochizuki and Trooshin \cite{MT19}. Some conditional uniform stability estimates in various $L_p$-norms were derived by Horv\'ath and Kiss \cite{HK10} for the inverse Sturm-Lioville problem by eigenvalues from several spectra. 

It is worth mentioning that the unconditional uniform stability of inverse spectral problems has been recently proved by Buterin \cite{But21} for integro-differential operators, by Buterin and Djuric \cite{BD22} for functional-differential operators with delay, and by Kuznetsova \cite{Kuz23} for the Sturm-Liouville-type operators with the so-called ``frozen argument''. However, the operators that are studied in \cite{But21, BD22, Kuz23} are non-local. Therefore, for solving the corresponding inverse problems, the classical methods (the Gelfand-Levitan method, the methods of spectral mappings, etc.) do not work, so the authors applied fundamentally different methods.

In this paper, we study the uniform stability of the inverse Sturm-Liouville problem in the non-self-adjoint case. First, we focus on the case of simple spectrum, and then consider the general case with multiplicities.
In the latter case, the approach of Buterin \cite{But07} is used for introducing the generalized spectral data. We find various subsets in the space of spectral data, on which the inverse mapping is Lipschitz continuous, and obtain the corresponding unconditional uniform stability estimates. Furthermore, the conditional uniform stability of the inverse problem under a priori restrictions on the potential is studied. In addition, we prove the uniform stability of the inverse problem by the so-called Cauchy data. Use of the Cauchy data is convenient for numerical reconstruction (see \cite{RS92, MPS94}) and for investigation of partial inverse problems (see \cite{Bond23}).

For investigating both conditional and unconditional uniform stability, we develop a general approach based on the method of spectral mappings \cite{FY01, Yur02}. This method consists in reducing a nonlinear inverse problem to a linear equation in a suitable Banach space. For the purposes of this paper, we present a modified construction of the main equation that allows us to analyze the uniform stability. Furthermore, we have found two new classes of spectral data, for which the main equation, as well as the inverse spectral problem, are uniquely solvable. For working with multiple eigenvalues, we develop the ideas of \cite{Bond20}, which allow us to obtain the uniform stability estimates for parameters of two problems having different eigenvalue multiplicities.

The method of spectral mappings is a universal tool in inverse spectral theory. By using this method, inverse problems for higher-order differential operators, for differential operators on graphs, for general first order systems, and for several other classes of differential operators have been solved (see, e.g., \cite{Yur02, Yur16} and references therein). Therefore, our approach, which is presented in this paper for the Sturm-Liouville problem \eqref{eqv}--\eqref{bc}, in the future can be applied to more general classes of differential operators, including self-adjoint and non-self-adjoint ones with multiple eigenvalues. 

The paper is organized as follows. In Section~\ref{sec:main}, we formulate the main results. In Sections~\ref{sec:maineq}--\ref{sec:cond}, the inverse Sturm-Liouville problem is studied for the case of simple eigenvalues. In Section~\ref{sec:maineq}, we derive the linear main equation, which is convenient for investigation of the uniform stability of the inverse spectral problem. In Section~\ref{sec:nsc}, the unique solvability of the main equation is proved for two new classes of spectral data in the non-self-adjoint case. In Section~\ref{sec:ubound}, we prove the uniform boundedness of the inverse problem solution. In Section~\ref{sec:ustab}, the uniform stability of the inverse problem is proved under upper bounds for the spectral data and for the norm of the inverse operator from the main equation. That theorem plays a crucial role for obtaining unconditional uniform stability in Section~\ref{sec:uncond} and conditional uniform stability in Section~\ref{sec:cond}. In Section~\ref{sec:mult}, we generalize our results to the case of multiple eigenvalues. Section~\ref{sec:Cauchy} is concerned with the uniform stability of the inverse problem by the Cauchy data.

Throughout the paper, we use the following \textbf{notations}:
\begin{itemize}
\item $\rho := \sqrt{\la}$, $\rho_n := \sqrt{\la_n}$, where $\arg \rho, \, \arg \rho_n \in \left[ -\frac{\pi}{2}, \frac{\pi}{2}\right)$, $n \ge 1$.
\item The prime $y'(x,\rho)$ denotes the derivative with respect to $x$ and the dot $\dot y(x, \rho)$, with respect to $\rho$.
\item In estimates, we denote by the same symbol $C$ various positive constants not depending on $x$, $\la$, $n$, $k$, etc. The notations $C = C(\Omega)$, $C = C(\Omega, K)$, etc. mean that $C$ depends on $\Omega$, on $\Omega$ and $K$, respectively, etc.
\item $\de_{jk}$ is the Kronecker delta.
\item We consider the standard functional spaces $L_2(0, \pi)$, $W_2^1[0,\pi]$ and the spaces $l_2$, $l_2^1$ of infinite sequences $X = \{ x_n \}_{n \ge 1}$ with the corresponding norms:
\begin{gather*}
\| f \|_{L_2} = \sqrt{\int_0^{\pi} |f(x)|^2 \, dx}, \qquad \| f \|_{W_2^1} = \sqrt{\| f \|_{L_2}^2 + \| f' \|_{L_2}^2}, \\
\| X \|_{l_2} = \sqrt{\sum_{n = 1}^{\infty} |x_n|^2}, \qquad
\| X \|_{l_2^1} = \sqrt{\sum_{n = 1}^{\infty} (n |x_n|)^2}.
\end{gather*}
\item The symbol $m$ denotes the Banach space of bounded infinite sequences $a = [a_{ni}]_{n \ge 1, \, i = 0, 1}$ with the norm $\| a \|_m = \sup_{n,i} |a_{ni}|$. The notation $\| . \|_{m \to m}$ is used for the norm in the space of bounded linear operators on $m$, and $I$ denotes the identity operator in $m$.
\item Along with the problem $L = L(q, h, H)$, we consider the problem $\tilde L = L(\tilde q, \tilde h, \tilde H)$ of the same form \eqref{eqv}--\eqref{bc} with the zero coefficients: $\tilde q(x) \equiv 0$, $\tilde h = \tilde H = 0$. We agree that, if a symbol $\gamma$ denotes an object related to $L$, then the symbol $\tilde \ga$ with tilde will denote the analogous object related to $\tilde L$. Similarly, the notations $L^{(1)} = L(q^{(1)}, h^{(1)}, H^{(1)})$ and $L^{(2)} = L(q^{(2)}, h^{(2)}, H^{(2)})$ are used for two problems of form \eqref{eqv}--\eqref{bc} with different coefficients. The upper index $^{(s)}$ means that the object belongs to the corresponding problem $L^{(s)}$, $s = 1, 2$.
\item Denote by $\mathcal S$ the set of sequences $S = \{ \rho_n, \al_n \}_{n \ge 1}$ such that $\rho_n, \al_n \in \mathbb C$, $\arg \rho_n \in [-\tfrac{\pi}{2}, \tfrac{\pi}{2})$, and the asymptotics \eqref{asympt} are valid for some $\om \in \mathbb C$ and $\{ \varkappa_n \}, \, \{ s_n \} \in l_2$. For two sequences $S^{(1)} = \{ \rho_n^{(1)}, \al_n^{(1)} \}_{n \ge 1}$ and $S^{(2)} = \{ \rho_n^{(2)}, \al_n^{(2)} \}_{n \ge 1}$ of $\mathcal S$, introduce the distance
\begin{equation} \label{defdist}
d(S^{(1)}, S^{(2)}) := \left( |\om^{(1)} - \om^{(2)}|^2 + \sum_{n \ge 1} \bigl( |\varkappa_n^{(1)} - \varkappa_n^{(2)}|^2 + |s_n^{(1)} - s_n^{(2)}|^2 \bigr) \right)^{1/2}
\end{equation}
\end{itemize}

\section{Main results} \label{sec:main}

Consider the boundary value problem $L = L(q, h, H)$ given by \eqref{eqv}--\eqref{bc}. It is well-known that $L$ has a countable set of eigenvalues $\{ \la_n \}_{n \ge 1}$. For convenience, introduce the spectral parameter $\rho := \sqrt{\la}$ and put $\rho_n := \sqrt{\la_n}$, $\arg \rho, \, \rho_n \in \left[ -\frac{\pi}{2}, \frac{\pi}{2}\right)$. Denote by $\varphi(x, \rho)$ the solution of equation \eqref{eqv} with $\la = \rho^2$ under the initial conditions $\vv(0, \rho) = 1$, $\vv'(0,\rho) = h$. Suppose that the eigenvalues $\{ \la_n \}_{n \ge 1}$ are simple. Then $\{ \varphi(x, \rho_n) \}_{n \ge 1}$ are the corresponding eigenfunctions, and the generalized weight numbers are correctly defined by the formula 
\begin{equation} \label{defal}
\al_n := \left( \int_0^{\pi} \varphi^2(x, \rho_n) \, dx\right)^{-1}.
\end{equation}

As in the self-adjoint case, the spectral data $\{ \rho_n, \al_n \}_{n \ge 1}$ satisfy the asymptotics \eqref{asympt} and uniquely specify $q$, $h$, and $H$ (see, e.g., \cite{FY01}). We begin with considering the corresponding inverse spectral problem for the case of simple eigenvalues:

\begin{ip} \label{ip:simp}
Given the spectral data $S = \{ \rho_n, \al_n \}_{n \ge 1}$, find $q$, $h$, and $H$.
\end{ip}

For solving Inverse Problem~\ref{ip:simp}, we derive the modified main equation $(I + \tilde R(x)) \psi(x) = \tilde \psi(x)$ by the method of spectral mappings. It is a linear equation in the Banach space $m$ for each fixed $x \in [0,\pi]$, $\psi(x)$ and $\tilde \psi(x)$ are elements of $m$, and $\tilde R(x)$ is a linear bounded operator in $m$ (see Section~\ref{sec:maineq} for details). The element $\tilde \psi(x)$ and $\tilde R(x)$ are constructed by using the spectral data $S$, and $\psi(x)$ is the unknown element, which is related to the problem parameters $q$, $h$, and $H$. Using the modified main equation, we obtain the following necessary and sufficient conditions for solvability of Inverse Problem~\ref{ip:simp}.

\begin{thm} \label{thm:nsc}
For complex numbers $\{ \rho_n, \al_n \}_{n \ge 1}$, $\arg \rho_n \in [-\frac{\pi}{2}, \tfrac{\pi}{2})$, to be the spectral data of a problem $L = L(q, h, H)$ with simple eigenvalues, the asymptotics \eqref{asympt} and the existence of a bounded inverse operator $(I + \tilde R(x))^{-1}$ for each fixed $x \in [0, \pi]$ are necessary and sufficient.
\end{thm}

Next, we find two new classes of spectral data, for which the invertibility of the operator $(I + \tilde R(x))$ is guaranteed.

\begin{thm} \label{thm:suff}
Suppose that $\{ \rho_n, \al_n \}_{n \ge 1} \in \mathcal S$, $\{ \rho_n^2 \}_{n \ge 1}$ are real and distinct, and $\{ \al_n \}_{n \ge 1}$ fulfill one of the following conditions: 
\begin{enumerate}
\item $\al_n \ne 0$, $\arg \al_n \in [\al_{min}, \al_{max}]$, $n \ge 1$, where $\al_{max} - \al_{min} < \pi$.
\item There exists $\mathcal Z \subset \mathbb N$ such that $\al_n < 0$ for $n \in \mathcal Z$ and, for $n \in \mathbb N \setminus \mathcal Z$, the values $\{ \mbox{Im}\, \al_n \}$ are either all positive or all negative.
\end{enumerate}
Then, $\{ \rho_n, \al_n \}_{n \ge 1}$ are the spectral data of some problem $L(q, h, H)$.
\end{thm}

Condition 1 of Theorem~\ref{thm:suff} holds, in particular, when $\al_n > 0$ for all $n \ge 1$. Therefore, Theorem~\ref{thm:suff} generalizes the classical result by Gelfand and Levitan \cite{GL51} to non-self-adjoint Sturm-Liouville operators.

Proceed to the uniform stability of Inverse Problem~\ref{ip:simp}. For a fixed $\Omega > 0$, define the balls in the space $\mathcal S$ of spectral data:
\begin{equation} \label{defB}
B_{\Omega} := \{ S \in \mathcal S \colon d(S, \tilde S) \le \Omega \}, \quad 
\mathring{B}_{\Omega} := \{ S \in B_{\Omega} \colon \omega = 0 \}.
\end{equation}
Note that due to the asymptotics \eqref{asympt}, $\om$ can be found from $\{ \rho_n \}_{n \ge 1}$ by the formula
$$
\om = \pi \lim_{n \to \infty} n (\rho_n - n + 1).
$$
Next, for each $S \in \mathcal S$, we can construct the operator $(I + \tilde R(x))$. For fixed $\Omega > 0$ and $K > 0$, introduce the set
\begin{equation} \label{defBK}
\mathring{B}_{\Omega, K} := \left\{ S \in \mathring{B}_{\Omega} \colon \| (I + \tilde R(x))^{-1} \|_{m \to m} \le K \right\}.
\end{equation}

In particular, the condition in \eqref{defBK} means that, for $S \in \mathring{B}_{\Omega,K}$, the corresponding operator $(I + \tilde R(x))$ has a bounded inverse for each $x \in [0,\pi]$. Therefore, by Theorem~\ref{thm:nsc}, the inverse problem is uniquely solvable, that is, there exist $q \in L_2(0, \pi)$, $h,H \in \mathbb C$ such that $S$ are the spectral data of $L(q, h, H)$. Our central result on the uniform stability is the following theorem.

\begin{thm} \label{thm:uni}
Suppose that $S^{(1)}, S^{(2)} \in \mathring{B}_{\Omega, K}$. Then, the corresponding solutions $(q^{(1)}, h^{(1)}, H^{(1)})$ and $(q^{(2)}, h^{(2)}, H^{(2)})$ of the inverse problem satisfy the uniform estimate
\begin{equation} \label{uni}
\| q^{(1)} - q^{(2)} \|_{L_2} + |h^{(1)} - h^{(2)}| + |H^{(1)} - H^{(2)}| \le C d(S^{(1)}, S^{(2)}),
\end{equation}
where $C = C(\Omega, K)$. Thus, the mapping $S \mapsto (q, h, H)$ is Lipschitz continuous on $\mathring{B}_{\Omega, K}$.
\end{thm}

Finding specific formulas for the constant $C = C(\Omega, K)$ in Theorem~\ref{thm:uni} and in the subsequent theorems is technically complicated and outside the scope of this paper.

Relying on Theorem~\ref{thm:uni}, we obtain several results on unconditional and conditional uniform stability of Inverse Problem~\ref{ip:simp}.

For $\Omega > 0$ and $\de \in (0, 1)$, denote by $\mathcal V_{\Omega, \de}$ the set of sequences $S = \{ \rho_n, \al_n \}_{n \ge 1} \in B_{\Omega}$ satisfying the additional requirements
\begin{gather} \label{reqrho}
\rho_n \in \mathbb R, \quad \rho_{n+1} - \rho_n \ge \de, \quad |\al_n| \ge \de, \quad n \ge 1, \\ \label{reqalpha}
\sup_{n \ge 1} \arg \al_n - \inf_{n \ge 1} \arg \al_n \le \pi - \de.
\end{gather}

Let $\tau = \{ \tau_n \}_{n \ge 1} \in l_2^1$ be a sequence of positive numbers. For $\Omega > 0$ and $\tau$, denote by $\mathcal V^+_{\Omega, \tau}$ ($\mathcal V^-_{\Omega, \tau}$) the set of sequences $S \in B_{\Omega}$ satisfying \eqref{reqrho} and the following requirement: there exists a set $\mathcal Z$ such that 
\begin{equation} \label{reqalpha2}
\al_n < 0, \: n \in \mathcal Z, \qquad
\mbox{Im}\, \al_n \ge \tau_n \:\: (-\mbox{Im}\,\al_n \ge \tau_n), \quad n \not\in \mathcal Z.
\end{equation}
In view of the asymptotics \eqref{asympt}, such a set $\mathcal Z$ is finite.

Clearly, sequences $S$ in $\mathcal V_{\Omega, \de}$ and in $\mathcal V_{\Omega, \tau}^{\pm}$ fulfill the Conditions 1 and 2 of Theorem~\ref{thm:suff}, respectively. Therefore, Inverse Problem~\ref{ip:simp} is uniquely solvable on these sets. Moreover, we prove the following uniform stability theorem.

\begin{thm} \label{thm:V}
Suppose that $S^{(1)}, S^{(2)} \in \mathcal V_{\Omega, \de}$ $(\mathcal V^{\pm}_{\Omega, \tau})$. Then, the corresponding solutions $(q^{(1)}, h^{(1)}, H^{(1)})$ and $(q^{(2)}, h^{(2)}, H^{(2)})$ of Inverse Problem~\ref{ip:simp} satisfy the uniform estimate \eqref{uni}, where $C = C(\Omega, \de)$ $(C = C(\Omega, \tau))$. In other words, the mapping $S \mapsto (q, h, H)$ is Lipschitz continuous on $\mathcal V_{\Omega, \de}$ $(\mathcal V^{\pm}_{\Omega, \tau})$.
\end{thm}

Confining ourselves by positive values $\{ \al_n \}_{n \ge 1}$ in Theorem~\ref{thm:V}, we obtain the analogous result to Theorem 3.8 by Savchuk and Shkalikov \cite{SS10} for $q \in L_2(0,\pi)$. The Robin boundary conditions \eqref{bc} in this paper are different from the Dirichlet boundary conditions in \cite{SS10}, but this difference is not principal. Similarly to the results of \cite{SS10}, Theorem~\ref{thm:V} can be extended to potentials in the scale of Sobolev spaces $W_2^{\theta}$, $\theta > -1$, but this is outside the scope of this paper.

Proceed to the conditional uniform stability. For a fixed $Q > 0$, introduce the ball
\begin{equation} \label{defPQ}
P_Q := \left\{ (q, h, H) \in L_2(0,\pi) \times \mathbb C \times \mathbb C \colon \| q \|_{L_2} + |h| + |H| \le Q \right\}.
\end{equation}

For $Q > 0$ and $A > 0$, we consider the class $P_{Q, A}$ of triples $(q, h, H) \in P_Q$, satisfying the additional spectral conditions:
\begin{enumerate}
\item The eigenvalues of $L(q, h, H)$ are simple.
\item $|\al_n| \le A$ for $n \ge 1$.
\end{enumerate}

Note that the latter condition does not follow from $(q, h, H) \in P_Q$. This is shown by Example~\ref{ex:unbound}. 

\begin{thm} \label{thm:uniQA}
Suppose that $(q^{(1)}, h^{(1)}, H^{(1)})$ and $(q^{(2)}, h^{(2)}, H^{(2)})$ belong to $P_{Q, A}$. Then, the uniform stability estimate \eqref{uni} holds with $C = C(Q, A)$. Thus, the mapping $S \mapsto (q, h, H)$ is Lipschitz continuous on the preimage of $P_{Q, A}$.
\end{thm}

For the case of multiple eigenvalues, we consider the generalized spectral data related to singularities of the Weyl function $M(\la)$. It is defined as $M(\rho^2) := \Phi(0, \rho)$, where $\Phi(x, \rho)$ is the solution of equation \eqref{eqv} satisfying the boundary conditions
$$
\Phi'(0,\rho) - h \Phi(0,\rho) = 1, \quad \Phi'(\pi, \rho) + H \Phi(\pi, \rho) = 0.
$$

Denote by $\mathcal I$ the index set of distinct eigenvalues and by $m_n$, the multiplicity of $\la_n$:
$$
\mathcal I := \{ 1 \} \cup \{ n \ge 2 \colon \la_n \ne \la_k, \, k < n \}, \quad
m_n := \# \{ k \in \mathbb N \colon \la_k = \la_n \}.
$$
Without loss of generality, assume that multiple eigenvalues are consecutive: $\la_n = \la_{n + 1} = \dots = \la_{n + m_n-1}$ for $n \in \mathcal I$. Then, the Weyl function admits the following representation (see \cite{But07}):
\begin{equation} \label{Laurent}
M(\la) = \sum_{n \in \mathcal I} \sum_{\nu = 0}^{m_n - 1} \frac{\al_{n + \nu}}{(\la- \la_n)^{\nu + 1}},
\end{equation}
where $\{ \al_n \}_{n \ge 1}$ are some complex numbers. Due to the asymptotics \eqref{asympt}, the eigenvalues $\la_n$ are simple for sufficiently large indices $n > N$. The corresponding coefficients $\al_n$ for $n > N$ in \eqref{Laurent} are non-zero and coincide with the quantities defined by \eqref{defal}. Therefore, we use the same notation for the generalized weight numbers as before. By virtue of Theorem 2 in \cite{But07}, the generalized spectral data $\{ \rho_n, \al_n \}_{n \ge 1}$ uniquely specify $q$, $h$, and $H$. The asymptotics \eqref{asympt} also hold.

Our approach allows us to obtain uniform stability estimates for problems $L^{(1)}$ and $L^{(2)}$ with different multiplicities of eigenvalues for $n \le N$. For this purpose, we consider the circle
\begin{equation} \label{defGa}
\Gamma_N := \bigl\{ \rho \in \mathbb C \colon |\rho| = N - \tfrac{1}{2} \bigr\},
\end{equation}
which contains the values $\{ \rho_n \}_{n \le N}$. Introduce the distance
\begin{align} \nonumber
d_N(S^{(1)}, S^{(2)}) = & \, \max_{\rho \in \Gamma_N} \bigl|M^{(1)}(\rho^2) - M^{(2)}(\rho^2)\bigr| \\ \label{defdN} & + \left(|\om^{(1)} - \om^{(2)}|^2 + \sum_{n > N} \bigl( |\varkappa_n^{(1)} - \varkappa_n^{(2)}|^2 + |s_n^{(1)} - s_n^{(2)}|^2 \bigr) \right)^{1/2},
\end{align}
which generalizes \eqref{defdist}. The first term in the right-hand side of \eqref{defdN} allows us to compare data $S^{(1)}$ and $S^{(2)}$ with different eigenvalue multiplicities for $n \le N$.

We obtain uniform stability estimates of form
\begin{equation} \label{unimult}
\| q^{(1)} - q^{(2)} \|_{L_2} + |h^{(1)} - h^{(2)}| + |H^{(1)} - H^{(2)}| \le C d_N(S^{(1)}, S^{(2)})
\end{equation}
under various conditions. In particular, we get the following theorem on the conditional uniform stability.

\begin{thm} \label{thm:uniQ}
For every $Q > 0$, there exists $N = N(Q)$ such that, for any $(q^{(1)}, h^{(1)}, H^{(1)})$ and $(q^{(2)}, h^{(2)}, H^{(2)})$ in $P_Q$, the estimate \eqref{unimult} holds with $C = C(Q)$.
\end{thm}

Theorem~\ref{thm:uniQ}, in contrast to Theorem~\ref{thm:uniQA}, contains no restrictions on the spectral data.

For studying the case of multiple eigenvalues, we derive the main equation in a suitable Banach space, which contains a continuous part and a discrete part. On the one hand, our construction resembles the approach of \cite{Bond20}. Therein, instead of finite spectral data with multiplicities, the continuous function $M(\rho^2)$ on the circle $\Gamma_N$ is considered.
On the other hand, we use the modification, which is described in Section~\ref{sec:maineq} and convenient for investigation of uniform stability.

Furthermore, in the case of multiple eigenvalues, it is natural to study the inverse spectral problem by the so-called Cauchy data (see, e.g., \cite{RS92, MPS94, BCK20, XB23, Bond23}).
Introduce the functions
\begin{equation} \label{defDelta}
\Delta(\rho) := \vv'(\pi, \rho) + H \vv(\pi, \rho), \quad \Delta_0(\rho) := \psi'(\pi, \rho) + H \psi(\pi, \rho),
\end{equation}
where $\psi(x, \rho)$ is the solution of equation \eqref{eqv} under the initial conditions $\psi(0, \rho) = 0$, $\psi'(0, \rho) = 1$. These functions admit the following standard representations, which are obtained by using the transformation operators for $\vv(x, \rho)$ and $\psi(x, \rho)$ (see \cite{FY01, Mar11}): 
\begin{align} \label{intDel}
\Delta(\rho) & = -\rho \sin \rho \pi + \om \cos \rho \pi + \int_0^{\pi} \mathscr N(t) \cos \rho t \, dt, \\ \label{intDel0}
\Delta_0(\rho) & = \cos \rho \pi + \om_0 \frac{\sin \rho \pi}{\rho} + \frac{1}{\rho}\int_0^{\pi} \mathscr N_0(t) \sin \rho t \, dt, 
\end{align}
where
$$
\om_0 := H + \frac{1}{2} \int_0^{\pi} q(t) \, dt, \quad \mathscr N, \, \mathscr N_0 \in L_2(0, \pi).
$$

We call the collection $\{ \mathscr N, \mathscr N_0, \om, \om_0 \}$ \textit{the Cauchy data} of the problem $L$. Consider the following problem.

\begin{ip} \label{ip:Cauchy}
Given the Cauchy data $\{ \mathscr N, \mathscr N_0, \om, \om_0 \}$, find $q$, $h$, and $H$.
\end{ip}

One can easily see that the zeros of $\Delta(\rho)$ and $\Delta_0(\rho)$ coincide with the square roots of the eigenvalues for equation \eqref{eqv} under the boundary conditions \eqref{bc} and
$$
y(0) = 0, \quad y'(\pi) + H y(\pi) = 0,
$$
respectively. Therefore, Inverse Problem~\ref{ip:Cauchy} is equivalent to the classical two-spectra inverse problem by Borg \cite{Borg46} and also to Inverse Problem~\ref{ip:simp}. 
Nevertheless, in the non-self-adjoint case, it is more convenient to study the inverse problem by the Cauchy data to avoid working with multiplicities. Moreover, the use of the Cauchy data is natural for numerical reconstruction of the potential (see \cite{RS92, MPS94}), as well as for analysis of solvability and stability for various types of inverse spectral problems (see \cite{BCK20, XB23, Bond23}).

Relying on Theorem~\ref{thm:uniQ}, we prove the following theorem on the uniform stability of Inverse Problem~\ref{ip:Cauchy}.

\begin{thm} \label{thm:uniCauchy}
Suppose that $(q^{(1)}, h^{(1)}, H^{(1)})$ and $(q^{(2)}, h^{(2)}, H^{(2)})$ belong to $P_Q$. Then
\begin{multline*}
\| q^{(1)} - q^{(2)} \|_{L_2} + |h^{(1)} - h^{(2)}| + |H^{(1)} - H^{(2)}| \\ \le C(Q) \bigl(\| \mathscr N^{(1)} - \mathscr N^{(2)} \|_{L_2} + \| \mathscr N_0^{(1)} - \mathscr N_0^{(2)} \|_{L_2} + |\om^{(1)} - \om^{(2)}| + |\om_0^{(1)} - \om_0^{(2)}|\bigr).
\end{multline*}
Thus, the mapping $\{ \mathscr N, \mathscr N_0, \om, \om_0 \} \mapsto (q, h, H)$ is Lipschitz continuous on the preimage of $P_Q$.
\end{thm}

An advantage of Theorem~\ref{thm:uniCauchy} is its independence of eigenvalue multiplicities. Furthermore, in contrast to Theorem~\ref{thm:uniQ}, choosing the index $N$ is not needed, when we formulate stability in terms of Cauchy data.

\section{Modified main equation} \label{sec:maineq}

In this section, we reduce Inverse Problem~\ref{ip:simp} in the case of simple eigenvalues to a linear equation. Our approach is based on the infinite system of linear equations \eqref{mainvv}, which has been derived in \cite{FY01} by the method of spectral mappings. However, we apply to the system \eqref{mainvv} a transform that is different from the one in \cite{FY01}. As a result, we get the modified main equation \eqref{main}, which is convenient for the analysis of the uniform stability. Furthermore, for the operator $(I + \tilde R(x))$ from the main equation, we study the properties that are essential for investigation of the inverse problem solvability and uniform stability in the next sections.

Consider a problem $L = L(q, h, H)$ with a simple spectrum and the so-called model problem $\tilde L = L(0, 0, 0)$.
Clearly, the problem $\tilde L$ has the spectral data
\begin{equation} \label{sd0}
\tilde \rho_n = n - 1, \quad n \ge 1, \qquad \tilde \al_1 = \frac{1}{\pi}, \quad
\tilde \al_n = \frac{2}{\pi}, \quad n \ge 2,
\end{equation}
and $\tilde \varphi(x, \rho) = \cos \rho x$.
Define the auxiliary functions
\begin{equation} \label{defD}
D(x, \rho, \theta) = \int_0^x \vv(t, \rho) \vv(t, \theta) \, dt, \quad
\tilde D(x, \rho, \theta) = \int_0^x \cos \rho t \cos \theta t \, dt.
\end{equation}

Introduce the notations
\begin{equation} \label{not}
\left.
\begin{array}{l}
\rho_{n0} = \rho_n, \quad \rho_{n1} = \tilde \rho_n, \quad \al_{n0} = \alpha_n, \quad \al_{n1} = \tilde \al_n, \\
\vv_{ni}(x) = \varphi(x, \rho_{ni}), \quad \tilde \vv_{ni}(x) = \tilde \varphi(x, \rho_{ni}), \quad n \ge 1, \: i = 0, 1.\end{array}\quad \quad \right\}
\end{equation}

By using the contour integration, the following proposition has been proved in \cite[Section 1.6.1]{FY01}:

\begin{prop}[\hspace*{-3pt}\cite{FY01}] \label{prop:contour}
The following relations hold:
\begin{gather} \label{relvv}
\tilde \vv(x, \rho) = \vv(x, \rho) + \sum_{k = 1}^{\infty} \bigl(\al_{k0} \tilde D(x, \rho, \rho_{k0}) \vv_{k0}(x) - \al_{k1} \tilde D(x, \rho, \rho_{k1}) \vv_{k1}(x)\bigr), \\ \label{relD}
D(x, \rho, \theta) - \tilde D(x, \rho, \theta) + \sum_{k = 1}^{\infty} \bigl( \al_{k0} \tilde D(x, \rho, \rho_{k0}) D(x, \rho_{k0}, \theta) - \al_{k1} \tilde D(x, \rho, \rho_{k1}) D(x, \rho_{k1}, \theta) \bigr) = 0,
\end{gather}
where the series converge absolutely and uniformly with respect to $x \in [0, \pi]$ and $\rho, \theta$ on compact sets. In particular, the substitution $\rho = \rho_{ni}$ in \eqref{relvv} implies
\begin{equation} \label{mainvv}
\tilde \vv_{ni}(x) = \vv_{ni}(x) + \sum_{k = 1}^{\infty} (\al_{k0} \tilde D(x, \rho_{ni}, \rho_{k0}) \vv_{k0}(x) - \al_{k1} \tilde D(x, \rho_{ni}, \rho_{k1}) \vv_{k1}(x)), \quad n \ge 1, \, i = 0, 1.
\end{equation}
\end{prop}

For each fixed $x \in [0,\pi]$, the relations \eqref{mainvv} can be treated as an infinite linear system with respect to the unknowns $\{ \vv_{ni}(x) \}$. Using the solution of this system, one can find the solution of Inverse Problem~\ref{ip:simp} due to the following proposition (see Lemma 1.6.5 in \cite{FY01}):

\begin{prop}[\hspace*{-3pt}\cite{FY01}] \label{prop:recqhH}
The series
\begin{equation} \label{defeps}
\eps(x) := \sum_{k = 1}^{\infty} \bigl( \al_{k0} \vv_{k0}(x) \tilde \vv_{k0}(x) - \al_{k1} \vv_{k1}(x) \tilde \vv_{k1}(x) \bigr)
\end{equation}
converges in $W_2^1[0,\pi]$. Furthermore,
\begin{equation} \label{recqhH}
q(x) = -2\eps'(x), \quad h = -\eps(0), \quad H = \eps(\pi).
\end{equation}
\end{prop}

It is inconvenient to use the system \eqref{mainvv} for the investigation of the inverse spectral problem, because the series in \eqref{mainvv} converges only ``with brackets''. In order to achieve the absolute convergence of the series, Freiling and Yurko \cite{FY01} have transformed the system \eqref{mainvv} to the linear main equation in the Banach space of infinite bounded sequences. Here, we use another transform to obtain the so-called \textit{modified main equation}, which is more convenient for studying the uniform stability of the inverse problem.

Denote 
$$
\hat \rho_n := \rho_{n0} - \rho_{n1}, \quad T_n := \begin{bmatrix}
\hat \rho_n & 1 \\
0 & 1
\end{bmatrix}
$$
Then
$$
T_n^{-1} = \begin{bmatrix}
\hat \rho_n^{-1} & -\hat \rho_n^{-1} \\
0 & 1
\end{bmatrix} \quad \text{for} \:\: \rho_{n0} \ne \rho_{n1}.
$$
For changing variables in \eqref{mainvv}, introduce the functions
\begin{gather} \nonumber
\dot \vv_{n1}(x) := \frac{d}{d\rho} \vv(x, \rho)_{|\rho = \rho_{n1}}, \quad \dot D(x, \rho_{n1}, \rho_{kj}) := \frac{d}{d\rho} D(x, \rho, \rho_{kj})_{|\rho = \rho_{n1}}, \\ \label{defpsi}
\begin{bmatrix}
\psi_{n0}(x) \\
\psi_{n1}(x)
\end{bmatrix} := T_n^{-1}
\begin{bmatrix}
\vv_{n0}(x) \\
\vv_{n1}(x)
\end{bmatrix}, \quad \rho_{n0} \ne \rho_{n1}, \qquad
\begin{bmatrix}
\psi_{n0}(x) \\
\psi_{n1}(x)
\end{bmatrix}
= \begin{bmatrix}
\dot \vv_{n1}(x) \\
\vv_{n1}(x)
\end{bmatrix}, \quad \rho_{n0} = \rho_{n1}, \\ \label{defR1}
\begin{bmatrix}
R_{n0,k0}(x) & R_{n0,k1}(x) \\
R_{n1,k0}(x) & R_{n1,k1}(x)
\end{bmatrix} = 
T_n^{-1} \begin{bmatrix}
\al_{k0} D(x, \rho_{n0}, \rho_{k0}) & -\al_{k1} D(x, \rho_{n0}, \rho_{k1}) \\
\al_{k0} D(x, \rho_{n1}, \rho_{k0}) & -\al_{k1} D(x, \rho_{n1}, \rho_{k1})
\end{bmatrix} T_k, \quad \rho_{n0} \ne \rho_{n1}, \\ \label{defR2}
\begin{bmatrix}
R_{n0,k0}(x) & R_{n0,k1}(x) \\
R_{n1,k0}(x) & R_{n1,k1}(x)
\end{bmatrix} = 
\begin{bmatrix}
\al_{k0} \dot D(x, \rho_{n1}, \rho_{k0}) & -\al_{k1} \dot D(x, \rho_{n1}, \rho_{k1}) \\
\al_{k0} D(x, \rho_{n1}, \rho_{k0}) & -\al_{k1} D(x, \rho_{n1}, \rho_{k1})
\end{bmatrix} T_k, \quad \rho_{n0} = \rho_{n1}.
\end{gather}
Similarly, define the functions $\tilde \psi_{ni}(x)$ and $\tilde R_{ni,kj}(x)$ by using $\tilde \vv$ and $\tilde D$ instead of $\vv$ and $D$, respectively.

The following estimates can be obtained similarly to (1.6.29)--(1.6.30) in \cite{FY01}:
\begin{equation} \label{estpsiR}
|\psi_{ni}(x)|, \, |\tilde \psi_{ni}(x)| \le C, \quad |R_{ni,kj}(x)|, \, |\tilde R_{ni,kj}(x)| \le \frac{C \xi_k}{|n - k| + 1}, \quad n,k \ge 1,
\end{equation}
where the constant $C$ depends on the problem $L$.
We discuss the derivation of \eqref{estpsiR} in more detail in Sections~\ref{sec:ubound} and~\ref{sec:ustab}, where the uniform dependence for the constant $C$ on the parameters of the problem $L$ and on the spectral data $\{ \rho_n, \al_n \}_{n \ge 1}$ is analyzed.

The following lemma shows the reduction of conditionally convergent series in Proposition~\ref{prop:contour} to absolutely convergent series.

\begin{lem}
The following relations hold:
\begin{gather} \label{mainpsi}
\tilde \psi_{ni}(x) = \psi_{ni}(x) + \sum_{k = 1}^{\infty} \sum_{j = 0}^{1} \tilde R_{ni,kj}(x) \psi_{kj}(x), \quad n \ge 1, \: i = 0,1, \\ \label{mainR}
R_{ni,kj}(x) - \tilde R_{ni,kj}(x) + \sum_{l = 1}^{\infty} \sum_{s = 0}^1 \tilde R_{ni,ls}(x) R_{ls,kj}(x) = 0, \quad n,k \ge 1, \: i,j = 0, 1,
\end{gather}
where the series converge absolutely and uniformly with respect to $x \in [0,\pi]$.
\end{lem}

\begin{proof}
The relations \eqref{mainvv} can be rewritten as follows:
$$
\begin{bmatrix}
\tilde \vv_{n0}(x) \\ \tilde \vv_{n1}(x)
\end{bmatrix} = 
\begin{bmatrix}
\vv_{n0}(x) \\ \vv_{n1}(x) 
\end{bmatrix} + \sum_{k = 1}^{\infty} \begin{bmatrix}
\al_{k0} \tilde D(x, \rho_{n0}, \rho_{k0}) & -\al_{k1} \tilde D(x, \rho_{n0}, \rho_{k1}) \\
\al_{k0} \tilde D(x, \rho_{n1}, \rho_{k0}) & -\al_{k1} \tilde D(x, \rho_{n1}, \rho_{k1})
\end{bmatrix} \begin{bmatrix} \vv_{k0}(x) \\ \vv_{k1}(x) \end{bmatrix}, \quad n \ge 1.
$$
It follows from \eqref{defpsi} that 
\begin{equation} \label{vvpsi}
\begin{bmatrix} \vv_{k0}(x) \\ \vv_{k1}(x) \end{bmatrix} = T_k \begin{bmatrix} \psi_{k0}(x) \\ \psi_{k1}(x) \end{bmatrix}
\end{equation}
for the both cases $\rho_{k0} \ne \rho_{k1}$ and $\rho_{k0} = \rho_{k1}$. Consequently, taking the notations \eqref{defpsi} and \eqref{defR1} into account, we directly arrive at the relation
\begin{equation} \label{matrpsi}
\begin{bmatrix}
\tilde \psi_{n0}(x) \\ \tilde \psi_{n1}(x)
\end{bmatrix} = 
\begin{bmatrix}
\psi_{n0}(x) \\ \psi_{n1}(x) 
\end{bmatrix} + \sum_{k = 1}^{\infty} \begin{bmatrix}
\tilde R_{n0, k0}(x) & \tilde R_{n0,k1}(x) \\
\tilde R_{n1,k0}(x) & \tilde R_{n1,k1}(x)
\end{bmatrix} \begin{bmatrix} \psi_{k0}(x) \\ \psi_{k1}(x) \end{bmatrix}
\end{equation}
for $\rho_{n0} \ne \rho_{n1}$. For the case $\rho_{n0} = \rho_{n1}$, we get by differentiation of \eqref{relvv} that
$$
\dot{\tilde \vv}_{n1}(x) = \dot \vv_{n1}(x) + \sum_{k = 1}^{\infty} \bigl(\al_{k0} \dot{\tilde D}(x, \rho_{n1}, \rho_{k0}) \vv_{k0}(x) - \al_{k1} \dot{\tilde D}(x, \rho_{n1}, \rho_{k1}) \vv_{k1}(x)\bigr).
$$
The latter relation together with \eqref{mainvv} and the definitions \eqref{defpsi}, \eqref{defR1} imply \eqref{matrpsi} for $\rho_{n0} = \rho_{n1}$. The relation \eqref{mainR} is proved analogously by using \eqref{relD}, \eqref{defR1}, and \eqref{defR2}.
The absolute convergence of the series in \eqref{mainpsi} and \eqref{mainR} follows from the estimates \eqref{estpsiR}.
\end{proof}

Consider the Banach space $m$ of infinite bounded sequences $a = [a_{ni}]_{n \ge 1, \, i = 0, 1}$ with the norm $\| a \|_m = \sup_{n,i} |a_{ni}|$. Let $x \in [0,\pi]$ be fixed. Define the infinite vector
\begin{equation} \label{vecpsi}
\psi(x) = [\psi_{10}(x), \psi_{11}(x), \psi_{20}(x), \psi_{21}(x), \dots]
\end{equation}
and the linear operator $R(x)$ acting on an element $a \in m$ by the rule
\begin{equation} \label{operR}
(R(x) a)_{ni} = \sum_{k = 1}^{\infty} \sum_{j = 0}^1 R_{ni,kj}(x) a_{kj}, \quad n \ge 1, \, i = 0, 1.
\end{equation}
Analogously, define $\tilde \psi(x)$ and $\tilde R(x)$.
According to the estimates \eqref{estpsiR}, for each fixed $x \in [0,\pi]$, the vectors $\psi(x)$ and $\tilde \psi(x)$ belong to $m$ and the operators $R(x)$ and $\tilde R(x)$ are bounded from $m$ to $m$: 
\begin{equation} \label{estopR}
\| R(x) \|_{m \to m} = \sup_{n,i} \sum_{k,j} |R_{ni,kj}(x)| \le C \sup_n \sum_{k = 1}^{\infty} \frac{\xi_k}{|n-k| + 1} < \infty.
\end{equation}
Moreover, $R(x)$ is continuous with respect to $x \in [0,\pi]$ in the operator norm.

Thus, the relations \eqref{mainpsi} and \eqref{mainR} can be rewritten as follows:
$$
\tilde \psi(x) = (I + \tilde R(x)) \psi(x), \quad 
R(x) - \tilde R(x) + \tilde R(x) R(x) = 0,
$$
where $I$ is the identity operator in $m$. The second relation implies
$$
(I + \tilde R(x))(I - R(x)) = I,
$$
which allows us to find the inverse operator to $(I + \tilde R(x))$.
As a result, we arrive at the following theorem.

\begin{thm} \label{thm:main}
For each fixed $x \in [0,\pi]$, the vector $\psi(x) \in m$ satisfies the equation
\begin{equation} \label{main}
\tilde \psi(x) = (I + \tilde R(x)) \psi(x), \quad x \in [0,\pi],
\end{equation}
in the Banach space $m$. Moreover, the operator $(I + \tilde R(x))$ has a bounded inverse, which is given by the formula
\begin{equation} \label{invR}
(I + \tilde R(x))^{-1} = I - R(x),
\end{equation}
so equation \eqref{main} is uniquely solvable.
\end{thm}

\begin{remark}
Note that $\psi_{ni}(x)$ and $R_{ni,kj}(x)$, which are defined by \eqref{defpsi}, \eqref{defR1}, and \eqref{defR2}, are continuous as $\rho_{n0} \to \rho_{n1}$. For example,
$$
\psi_{n0}(x) = \frac{\vv(x, \rho_{n0}) - \vv(x, \rho_{n1})}{\rho_{n0} - \rho_{n1}} \to \dot \vv(x, \rho)_{|\rho = \rho_{n1}}, \quad \rho_{n0} \to \rho_{n1}.
$$
This feature is important for our analysis of the uniform stability. In previous literature on the method of spectral mappings (see, e.g., \cite[Section 1.6]{FY01}), the change of variables was introduced in another way, with discontinuity at $\rho_{n0} = \rho_{n1}$:
$$
\begin{bmatrix}
\psi_{n0}(x) \\ \psi_{n1}(x)
\end{bmatrix} = \begin{bmatrix}
\chi_n & -\chi_n \\
0 & 1
\end{bmatrix} \begin{bmatrix}
\vv_{n0}(x) \\ \vv_{n1}(x)
\end{bmatrix}, \quad \chi_n := \begin{cases}
                                    \xi_n^{-1}, & \xi_n \ne 0, \\
                                    0, & \xi_n = 0,
                                \end{cases},
$$
where $\xi_n := |\rho_n - \tilde \rho_n| + |\al_n - \tilde \al_n|$. However, the invertibility of the operator $(I + \tilde R(x))$ in this paper is equivalent to the invertibility of the operator $(E + \tilde H(x))$ in \cite{FY01}, because it actually means the unique solvability of the system \eqref{mainvv}. Furthermore, the both operators $\tilde R(x)$ and $\tilde H(x)$ possess the approximation property and their components fulfill the estimate \eqref{estpsiR}. Therefore, their invertibility can be investigated similarly. 
\end{remark}

Analogously to Theorem 1.6.3 in \cite{FY01}, we obtain Theorem~\ref{thm:nsc} on the
necessary and sufficient conditions for the solvability of the inverse problem in the non-self-adoint case.

\begin{proof}[Proof of Theorem~\ref{thm:nsc}]
By necessity, the existence of the bounded inverse operator $(I + \tilde R(x))^{-1}$ follows from the relation \eqref{invR} in Theorem~\ref{thm:main}. 
By sufficiency, we require that the bounded operator $(I + \tilde R(x))^{-1}$ exists. Then, the main equation \eqref{main} has the unique solution $\psi(x) \in m$ for each fixed $x \in [0,\pi]$. Using its components $\psi_{ni}(x)$, one can construct the functions $\vv_{ni}(x)$ by \eqref{vvpsi} and find $q$, $h$, and $H$ via Proposition~\ref{prop:recqhH}. Relying on the asymptotics \eqref{asympt}, we get that the series for $q$ converges in $L_2(0,\pi)$ and for $h, \, H$, in $\mathbb C$, which concludes the proof.
\end{proof}

Note that the conditions $\rho_n \ne \rho_k$ ($n \ne k$) and $\al_n \ne 0$ are not required in Theorem~\ref{thm:nsc} by sufficiency, since they follow from the invertibility of the operator $(I + \tilde R(x))$ (see Lemma~\ref{lem:noninv}).

\section{Solvability of the main equation} \label{sec:nsc}

Let $S = \{ \rho_n, \al_n \}_{n \ge 1}$ be an arbitrary sequence of $\mathcal S$, not necessarily being spectral data for a certain problem $L$. In this section, we discuss the invertibility of the operator $(I + \tilde R(x))$ constructed by $S$ via formulas \eqref{defR1}, \eqref{defR2}, and \eqref{operR}. In particular, we prove Theorem~\ref{thm:suff}, which establishes the unique solvability of the inverse problem for two new classes of spectral data.

For investigating the invertibility of the operator $(I + \tilde R(x))$, we rely on the following auxiliary lemma.

\begin{lem} \label{lem:homo}
Let $x \in [0,\pi]$ be fixed. Then, the operator $(I + \tilde R(x))$ has a bounded inverse in $m$ if and only if the infinite system of equations
\begin{equation} \label{sysgamma}
\ga_{ni} + \sum_{k = 1}^{\infty} \bigl(\al_{k0} \tilde D(x, \rho_{ni}, \rho_{k0}) \ga_{k0} - \al_{k1} \tilde D(x, \rho_{ni}, \rho_{k1}) \ga_{k1} \bigr) = 0, \quad n \ge 1, \, i = 0, 1,
\end{equation}
has the only zero solution $\{ \ga_{ni} \}$ satisfying the estimates
$$
|\ga_{ni}| \le 1, \quad |\ga_{n0} - \ga_{n1}| \le |\hat \rho_n|, \quad n \ge 1, \, i = 0, 1.
$$
\end{lem}

\begin{proof}
According to the definition \eqref{operR} and the estimate \eqref{estpsiR} for $\tilde R_{ni,kj}(x)$, the operator $\tilde R(x)$ is approximated by the finite-dimensional operators $\tilde R^N(x)$, $N \ge 1$, given by
$$
(\tilde R^N(x) a)_{ni} = \sum_{k = 1}^N \sum_{j = 0}^1 \tilde R_{ni,kj}(x) a_{kj}.
$$
Therefore, by the Fredholm Theorem, the operator $(I + \tilde R(x))$ has a bounded inverse in $m$ if and only if the homogeneous system
\begin{equation} \label{homo}
\be_{ni} + \sum_{k = 1}^{\infty} \sum_{j = 0}^1 \tilde R_{ni,kj}(x) \be_{kj} = 0, \quad n \ge 1, \, i = 0, 1,
\end{equation}
has only the zero solution $\{ \be_{ni} \}$ in $m$: $|\be_{ni}| \le C$. The change of variables
$$
\begin{bmatrix}
\ga_{n0} \\ \ga_{n1}
\end{bmatrix} = 
T_n \begin{bmatrix}
\be_{n0} \\ \be_{n1}
\end{bmatrix}
$$
transforms the system \eqref{homo} to \eqref{sysgamma}.
\end{proof}

\begin{lem} \label{lem:noninv}
Suppose that $\rho_p = \rho_s$ for some $p \ne s$ or $\al_r = 0$ for some $r \ge 1$. Then, the operator $(I + \tilde R(\pi))$ is non-invertible in $m$.
\end{lem}

\begin{proof}
For each case, let us find a non-trivial solution of the system \eqref{sysgamma} for $x = \pi$. By virtue of \eqref{sd0} and \eqref{defD}, we have
$$
\tilde D(\pi, \rho_{n1}, \rho_{k1}) = \de_{nk}, \quad n \ge 1,
$$
where $\de_{nk}$ is the Kronecker delta.
Consequently, the equations \eqref{sysgamma} for $i = 1$ take the form
\begin{equation} \label{gamma1}
\sum_{k = 1}^{\infty} \al_{k0} \tilde D(\pi, \rho_{n1}, \rho_{k0}) \ga_{k0} = 0.
\end{equation}

\smallskip

\textit{Case 1:} $\rho_p = \rho_s$ for $p \ne s$. Put $\ga_{n0} = 0$ for all $n \ge 1$. Then, \eqref{gamma1} is satisfied and \eqref{sysgamma} for $i = 0$ takes the form
\begin{equation} \label{gamma0}
\sum_{k = 1}^{\infty} \al_{k1} \tilde D(\pi, \rho_{n0}, \rho_{k1}) \ga_{k1} = 0, \quad n \ge 1.
\end{equation}
Using \eqref{defD}, we transform \eqref{gamma0} as follows:
\begin{equation} \label{intgamma0}
\int_0^{\pi} \ga(x) \cos \rho_n x \, dx = 0, \quad n \ge 1, \qquad
\ga(x) := \sum_{k = 1}^{\infty} \tilde \al_k \ga_{k1} \cos \tilde \rho_k x.
\end{equation}
Due to our assumption, the sequence $\{ \cos \rho_n x \}_{n \ge 1}$ is incomplete in $L_2(0, \pi)$, so there exists a non-trivial function $\ga(x) \in L_2(0, \pi)$ satisfying \eqref{intgamma0}, and $\{ \ga_{k1} \}_{k \ge 1}$ can be found as the Fourier coefficients of $\ga(x)$ with respect to the orthogonal basis $\{ \cos \tilde \rho_k x \}_{k \ge 1}$.

\smallskip

\textit{Case 2:} the values $\{ \rho_n \}_{n \ge 1}$ are all distinct but $\al_r = 0$ for some $r \ge 1$. Put $\ga_{n0} = \de_{nr}$, $n \ge 1$. Then, \eqref{gamma1} is satisfied and the relations \eqref{sysgamma} for $i = 0$ take the form
\begin{equation*}
\sum_{k = 1}^{\infty} \al_{k1} \tilde D(\pi, \rho_{n0}, \rho_{k1}) \ga_{k1} = \de_{nr}, \quad n \ge 1.
\end{equation*}
The non-trivial solution $\{ \ga_{k1} \}_{k \ge 1}$ of this system can be uniquely found by the method similar to the Case 1.
\end{proof}

\begin{proof}[Proof of Theorem~\ref{thm:suff}]
Let us use Lemma~\ref{lem:homo} to prove the invertibility of the operator $(I + \tilde R(x))$. Fix $x \in [0,\pi]$ and consider a solution $\{ \ga_{ni} \}$ of the system \eqref{sysgamma}. Following the proof of Lemma 1.6.6 in \cite{FY01}, introduce the functions
\begin{align*}
\ga(\la) & := -\sum_{k = 1}^{\infty} \bigl( \al_{k0} \tilde D(x, \sqrt{\la}, \rho_{k0}) \ga_{k0} - \al_{k1} \tilde D(x, \sqrt{\la}, \rho_{k1}) \ga_{k1}\bigr), \\
G(\la) & := -\sum_{k = 1}^{\infty} \bigl( \al_{k0} \tilde E(x, \sqrt{\la}, \rho_{k0}) \ga_{k0} - \al_{k1} \tilde E(x, \sqrt{\la}, \rho_{k1}) \ga_{k1} \bigr),
\end{align*}
where
$$
\tilde E(x, \rho, \theta) := \frac{\tilde \Phi(x, \rho) \tilde \vv'(x, \theta) - \tilde \Phi'(x, \rho) \tilde \vv(x, \theta)}{\rho^2 - \theta^2}, \quad
\tilde \Phi(x, \rho) := \frac{\sin \rho(\pi - x)}{\rho \cos \rho \pi}.
$$

Applying the Residue Theorem to the function $\mathscr B(\la) := \overline{\ga(\overline{\la})} G(\la)$ on the contours $\Theta_N := \bigl\{ \la \colon |\la| = (N - 1/2)^2 \bigr\}$, we derive the relation
\begin{equation} \label{sumgamma}
\lim_{N \to \infty} \frac{1}{2\pi i} \oint_{\Theta_N} \mathscr B(\la) \, d\la = \sum_{n = 1}^{\infty} \al_{n0} |\ga_{n0}|^2 = 0.
\end{equation}

Let us show that the Conditions 1 and 2 of this theorem imply that $\ga_{n0} = 0$, $n \ge 1$, in \eqref{sumgamma}.
Under the Condition 1, the projections of $\al_n$ onto the bisector $\al = \tfrac{1}{2}(\al_{min} + \al_{max})$ are strictly positive, so the equality \eqref{sumgamma} is possible only if $\ga_{n0} = 0$ for all $n \ge 1$. Under the Condition 2, we get
$$
\mbox{Im} \left( \sum_{n = 1}^{\infty} \al_{n0} |\ga_{n0}|^2\right) = \sum_{n \in \mathbb N \setminus \mathcal Z} \mbox{Im}\, \al_n |\ga_{n0}|^2 = 0.
$$
Hence $\ga_{n0} = 0$ for all $n \in \mathbb N \setminus \mathcal Z$. Therefore, it remains from \eqref{sumgamma} that $\sum\limits_{n \in \mathcal Z} \al_n |\ga_{n0}|^2 = 0$. Since $\al_n < 0$, then $\ga_{n0} = 0$ for $n \in \mathcal Z$.

Proceeding similarly to the proof of Lemma 1.6.6 in \cite{FY01}, we conclude that the system \eqref{sysgamma} has only the trivial solution. Thus, Lemma~\ref{lem:homo} and Theorem~\ref{thm:nsc} yield the claim.
\end{proof}

\section{Uniform boundedness} \label{sec:ubound}

In this section, we prove that the solution $(q, h, H)$ of Inverse Problem~\ref{ip:simp} is uniformly bounded with respect to $S \in \mathring{B}_{\Omega, K}$ for positive constants $\Omega$ and $K$. This is an important auxiliary step in investigation of the uniform stability for the inverse problem. We follow the method for recovery of $(q, h, H)$ described in Section~\ref{sec:maineq} and step-by-step obtain uniform estimates for functions participating in the reconstruction.

Note that the definition of $\mathring{B}_{\Omega}$ from \eqref{defB} can be equivalently represented as follows:
\begin{equation} \label{defBO}
\mathring{B}_{\Omega} := \left\{ S \in \mathcal S \colon \| \{ \xi_n \}_{n \ge 1} \|_{l_2^1} \le \Omega\right\},
\end{equation}
where $\xi_n := \sqrt{|\rho_n - \tilde \rho_n|^2 + |\al_n - \tilde \al_n|^2}$. 

Suppose that $\Omega > 0$, and let $S = \{ \rho_n, \al_n \}_{n \ge 1}$ be any sequence in $\mathring{B}_{\Omega}$. Then, by using $S$ and the problem $\tilde L$, one can construct the functions $\tilde \psi_{ni}(x)$, $\tilde R_{ni,kj}(x)$ for $n, k \ge 1$, $i, j = 0, 1$ by \eqref{not} and the formulas analogous to \eqref{defpsi}, \eqref{defR1}-\eqref{defR2}, respectively.

\begin{lem} \label{lem:estt}
For $S \in \mathring{B}_{\Omega}$, the following estimates hold:
\begin{gather} \label{estvvt}
|\tilde \vv_{ni}(x)| \le C, \quad |\tilde \vv_{n0}(x) - \tilde \vv_{n1}(x)| \le C \xi_n, \quad
|\tilde \vv_{ni}'(x)| \le Cn, \quad |\tilde \vv'_{n0}(x) - \tilde \vv'_{n1}(x)| \le Cn \xi_n, \\ \label{estpsit}
|\tilde \psi_{ni}(x)| \le C, \quad |\tilde \psi_{ni}'(x)| \le C n, \\
\label{estRt}
|\tilde R_{ni,kj}(x)| \le \frac{C \xi_k}{|n - k| + 1}, \quad |\tilde R_{ni,kj}'(x)| \le C \xi_k,
\end{gather}
where $C = C(\Omega)$, $n,k \ge 1$, $i,j = 0, 1$, and $x \in [0,\pi]$. Consequently,
$$
\| \tilde \psi(x) \|_m \le C(\Omega), \quad \| \tilde R(x) \|_{m \to m} \le C(\Omega), \quad x \in [0,\pi].
$$
\end{lem}

\begin{proof}
The estimates \eqref{estvvt} are easily obtained by using the relations
$$
\tilde \vv_{ni}(x) = \cos \rho_{ni} x, \quad \tilde \vv'_{ni}(x) = -\rho_{ni} \sin \rho_{ni} x, \quad |\rho_{n0} - \rho_{n1}| \le \xi_n.
$$

Next, due to \eqref{defpsi}, we have $\tilde \psi_{ni}(x) = \tilde g_n(x, \rho_n)$, where the function
\begin{equation} \label{defgn}
\tilde g_n(x, \rho) := \frac{\cos \rho x - \cos \tilde \rho_n x}{\rho - \tilde \rho_n}
\end{equation}
is entire in $\rho$. Note that
\begin{equation} \label{estgn}
|\tilde g_n(x, \rho)| \le C(\Omega), \quad |\tilde g_n'(x, \rho)| \le C(\Omega) n \quad \text{for}\:\: |\rho - \tilde \rho_n| \le \Omega, \quad x \in [0,\pi],
\end{equation}
where $C(\Omega)$ does not depend on $n$. In view of \eqref{defBO} and $S \in \mathring{B}_{\Omega}$, we have $|\rho_n - \tilde \rho_n| \le \Omega$ for all $n \ge 1$, so \eqref{estgn} readily implies \eqref{estpsit}.

Proceed to obtaining \eqref{estRt}. Due to \eqref{not}, \eqref{defR1}, and \eqref{defR2}, we have
\begin{equation} \label{defRt00}
\tilde R_{n0,k0}(x) = \al_k \tilde G_n(x, \rho_n, \rho_k) \hat \rho_k,
\end{equation}
where
\begin{equation} \label{defGn}
\tilde G_n(x, \rho, \theta) := \frac{\tilde D(x, \rho, \theta) - \tilde D(x, \tilde \rho_n, \theta)}{\rho - \tilde \rho_n}
\end{equation}
is an entire function in $\rho$ and $\theta$. The complex Taylor formula
$$
\tilde D(x, \rho, \theta) = \tilde D(x, \tilde \rho_n, \theta) + \frac{(\rho - \tilde \rho_n)}{2 \pi i} \oint_{\gamma_n} \frac{\tilde D(x, w, \theta) \, dw}{(w - \tilde \rho_n)(w - \rho)}
$$
implies
$$
\tilde G_n(x, \rho, \theta) = \frac{1}{2 \pi i} \oint_{\gamma_n} \frac{\tilde D(x, w, \theta) \, dw}{(w - \tilde \rho_n)(w - \rho)},
$$
where $\ga_n := \{ w \in \mathbb C \colon |w - \tilde \rho_n| = 2\Omega \}$ and $|\rho - \tilde \rho_n| \le \Omega$. For the function
$\tilde D(x, w, \theta) = \int_0^x \cos wt \cos \theta t \, dt$,
we use the following standard estimates (see \cite[Section 1.6.1]{FY01}):
$$
|\tilde D(x, w, \theta)| \le \frac{C}{|n - k| + 1}, 
\quad |\tilde D'(x, w, \theta)| \le C,
$$
where $C = C(\Omega)$ for $|w - \tilde \rho_n| \le 2\Omega$ and $|\theta - \tilde \rho_k| \le \Omega$, $n,k \ge 1$.
Consequently,
\begin{gather} \label{estGn}
|\tilde G_n(x, \rho, \theta)| \le \frac{C(\Omega)}{|n - k| + 1}, \quad
|\tilde G_n'(x, \rho, \theta)| \le C(\Omega), \\ \nonumber \text{for} \:\:
|\rho - \tilde \rho_n| \le \Omega, \, |\theta - \tilde \rho_k| \le \Omega, \quad n,k\ge 1.
\end{gather}
This together with \eqref{defRt00} and the estimates $|\al_n| \le C(\Omega)$, $|\hat \rho_k| \le \xi_k$ directly imply \eqref{estRt} for $i = j = 0$. The other cases can be studied analogously. 

According to \eqref{estopR} and \eqref{estRt}, we deduce 
\begin{multline*}
\| \tilde R(x) \|_{m \to m} \le C(\Omega) \sup_{n \ge 1} \sum_{k = 1}^{\infty} \frac{\xi_k}{|n - k| + 1} \\ \le C(\Omega) \sqrt{\sum_{k = 1}^{\infty} \xi_k^2} \cdot \sup_{n \ge 1} \sqrt{\sum_{k = 1}^{\infty} \frac{1}{(|n-k|+1)^2}} \le C(\Omega) \| \{ \xi_k \}\|_{l_2}.
\end{multline*}
In view of \eqref{defBO}, we have $\| \{ \xi_k \}\|_{l_2} \le \| \{ \xi_k \}\|_{l_2^1} \le \Omega$, which concludes the proof.
\end{proof}

Next, suppose that $S \in \mathring{B}_{\Omega, K}$. Then, the operator $(I + \tilde R(x))^{-1}$ exists and is bounded by the constant $K$ in the operator norm for each fixed $x \in [0,\pi]$.
Define the operator $R(x)$ from \eqref{invR}:
$$
R(x) := I - (I + \tilde R(x))^{-1}.
$$
Clearly, $\| R(x) \|_{m \to m} \le K + 1$.
Since $R(x)$ is linear, it acts by the rule \eqref{operR} with some components $R_{ni,kj}(x)$. Following the proof strategy of Lemma~1.6.7 in \cite{FY01}, we derive the estimates
\begin{equation} \label{estR}
|R_{ni,kj}(x)| \le C \xi_k \left( \tfrac{1}{|n-k|+1} + \eta_n\right), \quad
|R_{ni,kj}(x)| \le C \xi_k \left( \tfrac{1}{|n-k|+1} + \eta_k\right),
\quad
|R_{ni,kj}'(x)| \le C \xi_k, 
\end{equation}
where $n,k \ge 1$, $i,j = 0, 1$, $x \in [0,\pi]$, and
$$
C = C(\Omega, K), \quad
\eta_n := \sqrt{\sum_{s = 1}^{\infty} \frac{1}{s^2 (|s - n| + 1)^2}}, \quad \{ \eta_n \}_{n \ge 1} \in l_2.
$$

Next, we construct the solution of the main equation \eqref{main} by the formula
$$
\psi(x) := (I - R(x)) \tilde \psi(x), 
$$
which takes the following element-wise form:
\begin{equation} \label{findpsi}
\psi_{ni}(x) = \tilde \psi_{ni}(x) - \sum_{k,j} R_{ni,kj}(x) \tilde \psi_{kj}(x), \quad n \ge 1, \, i = 0, 1.
\end{equation}

Using \eqref{findpsi} together with the estimates \eqref{estpsit} and \eqref{estR}, we obtain
\begin{equation} \label{estpsi}
|\psi_{ni}(x)| \le C, \quad |\psi_{ni}(x) - \tilde \psi_{ni}(x)| \le C \eta_n, \quad
|\psi'_{ni}(x)| \le C n, \quad |\psi_{ni}'(x) - \tilde \psi_{ni}'(x)| \le C, 
\end{equation}
where $C = C(\Omega, K)$. Using $\psi_{ni}(x)$, we find $\vv_{ni}(x)$ by \eqref{vvpsi}. Then, we immediately get
\begin{equation} \label{estvv}
|\vv_{ni}(x)| \le C, \quad |\vv_{ni}(x) - \tilde \vv_{ni}(x)| \le C \eta_n, \quad
|\vv'_{ni}(x)| \le C n, \quad |\vv_{ni}'(x) - \tilde \vv_{ni}'(x)| \le C,
\end{equation}
where $C = C(\Omega, K)$.

Using \eqref{estvvt} and \eqref{estvv}, we obtain the uniform estimates for the series \eqref{defeps}:
$$
\| \eps \|_{W_2^1} \le C(\Omega, K).
$$
By virtue of \eqref{recqhH}, we conclude that
$$
\| q \|_{L_2} \le C(\Omega, K), \quad |h|, |H| \le C(\Omega, K).
$$

Introduce the set
$$
\mathring{P}_Q := \{ (q, h, H) \in P_Q \colon \om = 0 \},
$$
where $P_Q$ was defined in \eqref{defPQ} and $\om$, in \eqref{defom}.
The above arguments yield the following result.

\begin{thm} \label{thm:ubound}
For each $S \in \mathring{B}_{\Omega, K}$, there exists a unique triple $(q, h, H) \in \mathring{P}_Q$ such that $S$ are the spectral data of the problem $L(q, h, H)$, where $Q = C(\Omega, K)$.
\end{thm}

Thus, Theorem~\ref{thm:ubound} asserts that the solution of the inverse problem is uniformly bounded on $\mathring{B}_{\Omega, K}$.

\section{Uniform stability} \label{sec:ustab}

The goal of this section is to prove Theorem~\ref{thm:uni} on the uniform stability of the inverse spectral problem.

Consider two sequences $S^{(1)} = \{ \rho_n^{(1)}, \alpha_n^{(1)} \}_{n \ge 1}$ and $S^{(2)} = \{ \rho_n^{(2)}, \alpha_n^{(2)} \}_{n \ge 1}$ of $\mathring{B}_{\Omega}$. Denote $Z := d(S^{(1)}, S^{(2)})$. In other words,
\begin{equation} \label{defZ}
\zeta_n := \sqrt{|\rho_n^{(1)} - \rho_n^{(2)}|^2 + |\alpha_n^{(1)} - \alpha_n^{(2)}|^2}, \quad
Z := \sqrt{\sum_{n = 1}^{\infty} (n\zeta_n)^2}.
\end{equation}
Obviously, for $S^{(1)}, S^{(2)} \in \mathring{B}_{\Omega}$, we have $\{ \zeta_n \} \in l_2^1$ and so $Z < \infty$.

\begin{lem} \label{lem:estdift}
For $S^{(1)}, S^{(2)} \in \mathring{B}_{\Omega}$, the following estimates hold:
\begin{gather} \label{estdifvvt}
|\tilde \vv_{n0}^{(1)}(x) - \tilde \vv_{n0}^{(1)}(x)| \le C\zeta_n, \quad
\left| \tfrac{d}{dx} (\tilde \vv_{n0}^{(1)} - \tilde \vv_{n0}^{(2)})(x)\right| \le C n \zeta_n, \quad \tilde \vv_{n1}^{(1)}(x) \equiv \tilde \vv_{n1}^{(2)}(x), \\ \label{estdifpsit}
|\tilde \psi_{n0}^{(1)}(x) - \tilde \psi_{n0}^{(2)}(x)| \le C \zeta_n, \quad
\left| \tfrac{d}{dx} (\tilde \psi_{n0}^{(1)} - \tilde \psi_{n0}^{(2)})(x) \right| \le C n \zeta_n, \quad \tilde \psi_{n1}^{(1)}(x) \equiv \tilde \psi_{n1}^{(2)}(x), \\ \label{estdifRte}
|\tilde R_{ni,kj}^{(1)}(x) - \tilde R_{ni,kj}^{(2)}(x)| \le \frac{C (\zeta_k + \zeta_n \xi_k^{(1)})}{|n-k|+1}, \quad 
\left| \tfrac{d}{dx} (\tilde R_{ni,kj}^{(1)} - \tilde R_{ni,kj}^{(2)})(x)\right| \le C (\zeta_k + Z \xi_k^{(1)}),
\end{gather}
where $C = C(\Omega)$, $n,k \ge 1$, $i,j = 0, 1$, $x \in [0,\pi]$. Consequently,
\begin{equation} \label{estdifRt}
\| \tilde R^{(1)}(x) - \tilde R^{(2)}(x) \|_{m \to m} \le C(\Omega) \| \{ \zeta_n \} \|_{l_2} \le C(\Omega) Z.
\end{equation}
\end{lem}

\begin{proof}
The estimates \eqref{estdifvvt} for $\tilde \vv_{ni}^{(s)}(x) = \cos \rho_{ni}^{(s)} x$ are easily deduced by using the relations 
$$
\rho_n^{(s)} = n - 1 + \hat \rho_n^{(s)}, \quad |\hat \rho_n^{(s)}| \le \xi_n, \quad |\hat \rho_n^{(1)} - \hat \rho_n^{(2)}| \le \zeta_n, \quad n \ge 1, \, s = 1, 2.
$$

According to \eqref{defpsi} and \eqref{sd0}, we have
$\tilde \psi_{n0}^{(s)}(x) = \tilde g_n(x, \rho_n^{(s)})$, where the entire function $\tilde g_n(x, \rho)$ was defined in \eqref{defgn}. Using the uniform estimate \eqref{estgn} and Schwarz's lemma, we obtain
$$
|\tilde g_n(x, \rho_n^{(1)}) - \tilde g_n(x, \rho_n^{(2)})| \le C(\Omega) |\rho_n^{(1)} - \rho_n^{(2)}|, \quad n \ge 1,
$$
which implies the estimate \eqref{estdifpsit} for $|\tilde \psi_{n0}^{(1)}(x) - \tilde \psi_{n0}^{(2)}(x)|$.

Due to \eqref{defRt00}, we have
$$
\tilde R^{(1)}_{n0,k0}(x) - \tilde R^{(2)}_{n0,k0}(x) = \al_k^{(1)} \tilde G_n(x, \rho_n^{(1)}, \rho_k^{(1)}) \hat \rho_k^{(1)} - \al_k^{(2)} \tilde G(x, \rho_n^{(2)}, \rho_k^{(2)}) \hat \rho_k^{(2)},
$$
where the entire function $\tilde G_n(x, \rho, \theta)$ was defined in \eqref{defGn}.
Grouping of the terms implies
\begin{align*}
& \tilde R^{(1)}_{n0,k0}(x) - \tilde R^{(2)}_{n0,k0}(x) = T_1 + T_2 + T_3 + T_4, \\
& T_1 := \bigl( \al_k^{(1)} - \al_k^{(2)} \bigr) \tilde G_n(x, \rho_n^{(1)}, \rho_k^{(1)}) \hat \rho_k^{(1)}, \\
& T_2 := \al_k^{(2)} \bigl( \tilde G_n(x, \rho_n^{(1)}, \rho_k^{(1)}) - \tilde G_n(x, \rho_n^{(2)}, \rho_k^{(1)})\bigr) \hat \rho_k^{(1)}, \\
& T_3 := \al_k^{(2)} \bigl( \tilde G_n(x, \rho_n^{(2)}, \rho_k^{(1)}) - \tilde G_n(x, \rho_n^{(2)}, \rho_k^{(2)}) \bigr) \hat \rho_k^{(1)}, \\
& T_4 := \al_k^{(2)} \tilde G_n(x, \rho_n^{(2)}, \rho_k^{(2)}) \bigl(\rho_k^{(1)} - \rho_k^{(2)}\bigr).
\end{align*}

The uniform estimate \eqref{estGn} for $\tilde G_n(x, \rho, \theta)$ together with Schwarz's Lemma and \eqref{defZ} imply
\begin{align*}
\bigl|\tilde G_n(x, \rho_n^{(1)}, \rho_k^{(1)}) - \tilde G_n(x, \rho_n^{(2)}, \rho_k^{(1)})\bigr| & \le \frac{C |\rho_n^{(1)} - \rho_n^{(2)}|}{|n-k| + 1} \le \frac{C \zeta_n}{|n - k| + 1}, \\
\bigl|\tilde G_n(x, \rho_n^{(2)}, \rho_k^{(1)}) - \tilde G_n(x, \rho_n^{(2)}, \rho_k^{(2)})\bigr| & \le \frac{C |\rho_k^{(1)} - \rho_k^{(2)}|}{|n-k| + 1} \le \frac{C \zeta_k}{|n - k| + 1},
\end{align*}
where $C = C(\Omega)$ does not depend on $n,k \ge 1$ and $x \in [0,\pi]$. Using the latter estimates together with \eqref{estGn} and the following ones:
$$
\bigl|\al_k^{(1)} - \al_k^{(2)}\bigr| \le \zeta_k, \quad \bigl|\al_k^{(2)}\bigr| \le C(\Omega), \quad \bigl|\hat \rho_k^{(1)}\bigr| \le \xi_k^{(1)} \le \Omega, \quad 
\bigl| \rho_k^{(1)} - \rho_k^{(2)}\bigr| \le \zeta_k,
$$
we obtain
$$
|T_1|, \, |T_3|, \, |T_4| \le \frac{C \zeta_k}{|n - k| + 1}, \quad
|T_2| \le \frac{C \zeta_n \xi_k^{(1)}}{|n - k| + 1}, \quad C = C(\Omega).
$$
This yields \eqref{estdifRte} for the difference $\bigl( \tilde R_{ni,kj}^{(1)} - \tilde R_{ni,kj}^{(2)} \bigr)$ with $i = j = 0$. The other cases, as well as the derivatives $\tfrac{d}{dx}\tilde \psi_{ni}^{(s)}(x)$ and $\tfrac{d}{dx}\tilde R_{ni,kj}^{(s)}(x)$, can be considered similarly.

Finally,
\begin{align*}
\bigl\| \tilde R^{(1)}(x) - \tilde R^{(2)}(x) \bigr\|_{m \to m} & = \sup_{n,i} \sum_{k,j} \bigl|\tilde R_{ni,kj}^{(1)}(x) - \tilde R_{ni,kj}^{(2)}(x)\bigr| \le C(\Omega) \sup_{n \ge 1} \sum_{k = 1}^{\infty} \frac{\zeta_k + \zeta_n \xi_k^{(1)}}{|n -k| + 1} \\
& \le C(\Omega) \Bigl( \| \{ \zeta_k\} \|_{l_2} + \sup_{n \ge 1} |\zeta_n| \cdot \| \{ \xi_k^{(1)}\} \|_{l_2} \Bigr).
\end{align*}
Since $\sup_{n \ge 1} |\zeta_n| \le \| \{ \zeta_n\} \|_{l_2}$ and $\| \{ \xi_k^{(1)}\} \|_{l_2} \le \Omega$, this yields \eqref{estdifRt} and so concludes the proof.
\end{proof}

Let us additionally assume that $S^{(1)}$ and $S^{(2)}$ belong to $\mathring{B}_{\Omega, K}$ with some $K > 0$. Then, the operators $(I + \tilde R^{(s)}(x))$ are invertible in $m$ for $s = 1, 2$ and $x \in [0,\pi]$. Therefore, the operators $R^{(s)}(x)$, as well as the functions $\psi_{ni}^{(s)}(x)$ and $\vv_{ni}^{(s)}(x)$ can be constructed.

\begin{lem} \label{lem:estdif}
For $S^{(1)}, \, S^{(2)} \in \mathring{B}_{\Omega, K}$, the following estimates hold:
\begin{gather} \label{estdifR}
\bigl| R_{ni,kj}^{(1)}(x) - R_{ni,kj}^{(2)}(x)\bigr| \le C \bigl( \zeta_k + Z \xi_k^{(1)}\bigr) \left( \frac{1}{|n-k|+1} + \eta_n\right), \\ \label{estdifR2}
\bigl| R_{ni,kj}^{(1)}(x) - R_{ni,kj}^{(2)}(x)\bigr| \le C \bigl( \zeta_k + Z \xi_k^{(1)}\bigr) \left( \frac{1}{|n-k|+1} + \eta_k\right), \\ 
\label{estdifRp}
\bigl| \tfrac{d}{dx}(R_{ni,kj}^{(1)} - R_{ni,kj}^{(2)})(x)\bigr| \le C \bigl( \zeta_k + Z \xi_k^{(1)}\bigr), \\ \label{estdifpsi}
\bigl| \psi_{ni}^{(1)}(x) - \psi_{ni}^{(2)}(x)\bigr| \le C (\zeta_n + Z \eta_n), \quad
\bigl| \tfrac{d}{dx} (\psi_{ni}^{(1)} - \psi_{ni}^{(2)})(x)\bigr| \le C (n \zeta_n + Z), \\ \label{estdifvv}
\bigl| \vv_{ni}^{(1)}(x) - \vv_{ni}^{(2)}(x)\bigr| \le C (\zeta_n + Z \eta_n), \quad
\bigl| \tfrac{d}{dx} (\vv_{ni}^{(1)} - \vv_{ni}^{(2)})(x)\bigr| \le C (n \zeta_n + Z), \\ \label{estvvpsi}
\bigl| \vv_{n0}^{(1)}(x) - \vv_{n1}^{(1)}(x) - \vv_{n0}^{(2)}(x) + \vv_{n1}^{(2)}(x) + (\rho_n^{(2)} - \rho_n^{(1)})\psi_{n0}^{(2)}(x)\bigr| \le C (\zeta_n + Z \eta_n) \xi_n^{(1)}, \\ \label{estvvpsip}
\bigl| \tfrac{d}{dx} \bigl(\vv_{n0}^{(1)} - \vv_{n1}^{(1)} - \vv_{n0}^{(2)} + \vv_{n1}^{(2)} + (\rho_n^{(2)} - \rho_n^{(1)})\psi_{n0}^{(2)}\bigr)(x)\bigr| \le C (n \zeta_n + Z) \xi_n^{(1)},
\end{gather}
where $C = C(\Omega, K)$, $n,k \ge 1$, $i,j = 0,1$, $x \in [0,\pi]$.
\end{lem}

\begin{proof}
Throughout this proof, we mean that $C = C(\Omega, K)$ and omit the argument $(x)$ for brevity. All the obtained estimates are uniform with respect to $x \in [0,\pi]$.
Using \eqref{invR}, we get
\begin{equation*} 
R^{(1)} - R^{(2)} = (I + \tilde R^{(2)})^{-1} - (I + \tilde R^{(1)})^{-1} = (I + \tilde R^{(2)})^{-1} (\tilde R^{(1)} - \tilde R^{(2)}) (I + \tilde R^{(1)})^{-1}.
\end{equation*}
This together with the estimates \eqref{estdifRt} and $\| (I + \tilde R^{(s)})^{-1}\|_{m \to m} \le K$, $s = 1, 2$, imply
\begin{equation} \label{estdifRop}
\| R^{(1)} - R^{(2)} \|_{m \to m} \le C(\Omega, K) Z.
\end{equation}
It can be deduced from \eqref{invR} that
$$
R = \tilde R - R \tilde R, \quad R = \tilde R - \tilde R R.
$$
In the element-wise form, we get
$$
R_{ni,kj} = \tilde R_{ni,kj} - \sum_{l,s} R_{ni,ls} \tilde R_{ls,kj}, \quad
R_{ni,kj} = \tilde R_{ni,kj} - \sum_{l,s} \tilde R_{ni,ls} R_{ls,kj}.
$$
Hence
\begin{align} \label{reldifR1}
R_{ni,kj}^{(1)} - R_{ni,kj}^{(2)} & = \tilde R_{ni,kj}^{(1)} - \tilde R_{ni,kj}^{(2)} - \sum_{l,s} (R_{ni,ls}^{(1)} - R_{ni,ls}^{(2)}) \tilde R_{ls,kj}^{(1)} - \sum_{l,s} R_{ni,ls}^{(2)} (\tilde R_{ls,kj}^{(1)} - \tilde R_{ls,kj}^{(2)}), \\ \label{reldifR2}
R_{ni,kj}^{(1)} - R_{ni,kj}^{(2)} & = \tilde R_{ni,kj}^{(1)} - \tilde R_{ni,kj}^{(2)} - \sum_{l,s} (\tilde R_{ni,ls}^{(1)} - \tilde R_{ni,ls}^{(2)}) R_{ls,kj}^{(1)} - \sum_{l,s} \tilde R_{ni,ls}^{(2)} (R_{ls,kj}^{(1)} - R_{ls,kj}^{(2)}).
\end{align}
Using \eqref{estRt}, \eqref{estR}, \eqref{estdifRte}, and \eqref{estdifRop}, we obtain
\begin{align*}
& \left| \sum_{l,s} (R_{ni,ls}^{(1)} - R_{ni,ls}^{(2)}) \tilde R_{ls,kj}^{(1)} \right| \le \| R^{(1)} - R^{(2)} \|_{m \to m} \sup_{l,s} |\tilde R^{(1)}_{ls,kj}| \le CZ \xi_k^{(1)}, \\
& \left| \sum_{l,s} R_{ni,ls}^{(2)} (\tilde R_{ls,kj}^{(1)} - \tilde R_{ls,kj}^{(2)}) \right| \le \| R^{(2)} \|_{m \to m} \sup_{l,s} |\tilde R_{ls,kj}^{(1)} - \tilde R_{ls,kj}^{(2)}| \le C \zeta_k.
\end{align*}
Consequently, it follows from \eqref{reldifR1} that
\begin{equation} \label{estdifRrough}
\bigl| R_{ni,kj}^{(1)} - R_{ni,kj}^{(2)}\bigr| \le C (\zeta_k + Z \xi_k^{(1)}).
\end{equation}
Let us obtain a more precise estimate. Using \eqref{estRt}, \eqref{estR}, \eqref{estdifRte}, and \eqref{estdifRrough}, we get
\begin{align*}
& \left| \sum_{l,s} (\tilde R_{ni,ls}^{(1)} - \tilde R_{ni,ls}^{(2)}) R_{ls,kj}^{(1)} \right| \le C \sum_{l = 1}^{\infty} \frac{(\zeta_l + Z \xi_l^{(1)}) \xi_k^{(1)}}{|l - n| + 1} \le C Z \eta_n \xi_k^{(1)} \\
& \left| \sum_{l,s} \tilde R_{ni,ls}^{(2)} (R_{ls,kj}^{(1)} -  R_{ls,kj}^{(2)}) \right| \le C \sum_{l = 1}^{\infty} \frac{ \xi_l^{(2)} (\zeta_k + Z \xi_k^{(1)})}{|l - n| + 1} \le C \eta_n (\zeta_k + Z \xi_k^{(1)}).
\end{align*}
Using the latter estimates together with \eqref{estdifRte} in \eqref{reldifR2}, we arrive at \eqref{estdifR}. The estimates \eqref{estdifR2} and \eqref{estdifRp} are obtained analogously.

Next, using \eqref{findpsi}, we get
$$
\psi_{ni}^{(1)} - \psi_{ni}^{(2)} = \tilde \psi_{ni}^{(1)} - \tilde \psi_{ni}^{(2)} - \sum_{k,j} \bigl( R_{ni,kj}^{(1)} - R_{ni,kj}^{(2)}\bigr) \psi_{kj}^{(1)} - \sum_{k,j} R_{ni,kj}^{(2)} \bigl( \psi_{kj}^{(1)} - \psi_{kj}^{(2)}\bigr).
$$
This relation together with \eqref{estR}, \eqref{estpsi}, and the already proved estimates of Lemmas~\ref{lem:estt}, \ref{lem:estdift}, and \ref{lem:estdif} imply \eqref{estdifpsi}. The estimates \eqref{estdifvv} readily follow from \eqref{vvpsi}, \eqref{estdifpsi}, and $|\hat \rho_n^{(1)} - \hat \rho_n^{(2)}| \le \zeta_n$.

It remains to prove \eqref{estvvpsi} and \eqref{estvvpsip}. Suppose that $\rho_n^{(1)} \ne \tilde \rho_n$ and $\rho_n^{(2)} \ne \tilde \rho_n$ for some $n \ge 1$. Then, according to \eqref{not} and \eqref{defpsi}, the first relation in \eqref{estdifpsi} for $i = 0$ can be rewritten as follows:
$$
\left| \frac{\vv_{n0}^{(1)} - \vv_{n1}^{(1)}}{\rho_n^{(1)} - \tilde \rho_n} - \frac{\vv_{n0}^{(2)} - \vv_{n1}^{(2)}}{\rho_n^{(2)} - \tilde \rho_n}\right| \le C (\zeta_n + Z \eta_n).
$$
Multiplying the expression in the left-hand side by $(\rho_n^{(1)} - \tilde \rho_n)$ and using the estimate $|\rho_n^{(1)} - \tilde \rho_n| \le \xi_n^{(1)}$, we arrive at \eqref{estvvpsi}. The cases when $\rho_n^{(1)} = \tilde \rho_n$ or $\rho_n^{(2)} = \tilde \rho_n$ are considered analogously. Similarly, the estimate \eqref{estvvpsip} follows from the second relation in \eqref{estdifpsi} for $i = 0$.
\end{proof}

Using the sequences $\{ \vv_{ni}^{(s)}(x) \}$ and $\{ \tilde \vv_{ni}^{(s)}(x) \}$, one can construct the corresponding functions $\eps^{(s)}(x)$ for $s = 1, 2$ by formula \eqref{defeps}. The following lemma presents the estimate for their difference.

\begin{lem} \label{lem:esidifeps}
For $S^{(1)}, S^{(2)} \in \mathring{B}_{\Omega, K}$, the following estimate holds:
\begin{equation} \label{estdifeps}
\| \eps^{(1)} - \eps^{(2)} \|_{W_2^1} \le C(\Omega, K) Z.
\end{equation}
\end{lem}

\begin{proof}
Let us represent \eqref{defeps} in the form $\eps = \tilde \eps + \hat \eps$, where
\begin{align*}
\tilde \eps & := \sum_{k = 1}^{\infty} \bigl( \al_{k0} \tilde \vv_{k0}^2 - \al_{k1} \tilde\vv_{k1}^2\bigr), \\
\hat \eps & := \sum_{k = 1}^{\infty} \bigl( \al_{k0} \tilde \vv_{k0} (\vv_{k0} - \tilde \vv_{k0}) - \al_{k1} \tilde \vv_{k1} (\vv_{k1} - \tilde \vv_{k1})\bigr).
\end{align*}
Then
$$
\tilde \eps^{(1)} - \tilde \eps^{(2)} = \sum_{k = 1}^{\infty} \bigl( \al_{k0}^{(1)} (\tilde \vv_{k0}^{(1)})^2 - \al_{k0}^{(2)} (\tilde \vv_{k0}^{(2)})^2 \bigr) = \sum_{k = 1}^{\infty} \bigl( \al_k^{(1)} \cos^2 (\rho_k^{(1)} x) - \al_k^{(2)} \cos^2 (\rho_k^{(2)} x) \bigr).
$$
Differentiation implies
$$
\tfrac{d}{dx} \bigl( \tilde \eps^{(1)} - \tilde \eps^{(2)}\bigr)(x) = - \sum_{k = 1}^{\infty} \bigl(\al_k^{(1)} \rho_k^{(1)} \sin (2 \rho_k^{(1)} x) - \al_k^{(2)} \rho_k^{(2)} \sin (2 \rho_k^{(2)}x)\bigr).
$$
Therefore, using \eqref{asympt}, \eqref{defBO}, and \eqref{defZ}, we get
$$
\tfrac{d}{dx} \bigl( \tilde \eps^{(1)} - \tilde \eps^{(2)}\bigr)(x) = - \sum_{k = 1}^{\infty} \left( s_k^- \sin 2kx + \tfrac{4}{\pi} x \varkappa_k^- \cos 2kx + r_k(x) \right),
$$
where 
$$
s_k^- := s_k^{(1)} - s_k^{(2)}, \quad \varkappa_k^- := \varkappa_k^{(1)} - \varkappa_k^{(2)}, \quad \bigl\| \{ s_k^- \} \bigr\|_{l_2} \le Z, \quad \bigl\| \{ \varkappa_k^- \} \|_{l_2} \le Z,
$$
and the series of the remainder terms $\sum_{k \ge 1} r_k(x)$ converges absolutely and uniformly on $[0,\pi]$ to a continuous function bounded by $C(\Omega) Z$. Hence 
\begin{equation} \label{smeps1}
\| \tfrac{d}{dx} \bigl( \tilde \eps^{(1)} - \tilde \eps^{(2)} \bigr)\|_{L_2} \le C(\Omega) Z.
\end{equation}

Using the estimates for $\vv_{ni}^{(s)}$ and $\tilde \vv_{ni}^{(s)}$ from Lemmas~\ref{lem:estt}, \ref{lem:estdift}, \ref{lem:estdif} together with \eqref{estvv}, one can show that the series 
\begin{align*}
\eps^{(1)} - \eps^{(2)} & = \sum_{k = 1}^{\infty} \bigl( \al_{k0}^{(1)} \tilde \vv_{k0}^{(1)} \vv_{k0}^{(1)} - \al_{k0}^{(2)} \tilde \vv_{k0}^{(2)} \vv_{k0}^{(2)} - \al_{k1} \tilde \vv_{k1} (\vv_{k1}^{(1)} - \vv_{k1}^{(2)})\bigr), 
\\
\tfrac{d}{dx} \bigl( \hat \eps^{(1)} - \hat \eps^{(2)}\bigr) & =
\sum_{k = 1}^{\infty} \tfrac{d}{dx}\bigl( \al_{k0}^{(1)} \tilde \vv_{k0}^{(1)} (\vv_{k0}^{(1)} - \tilde \vv_{k0}^{(1)}) - \al_{k0}^{(2)} \tilde \vv_{k0}^{(2)} (\vv_{k0}^{(2)} - \tilde \vv_{k0}^{(2)}) - \al_{k1} \tilde \vv_{k1} (\vv_{k1}^{(1)} - \vv_{k1}^{(2)}) \bigr)
\end{align*}
converge absolutely and uniformly on $[0,\pi]$ and
\begin{equation} \label{smeps2}
\max_{x \in [0,\pi]}\left| \bigl(\eps^{(1)} - \eps^{(2)}\bigr)(x) \right| \le C(\Omega, K) Z, \quad
\max_{x \in [0,\pi]}\left| \tfrac{d}{dx} \bigl( \hat \eps^{(1)} - \hat \eps^{(2)}\bigr)(x) \right| \le C(\Omega, K) Z.
\end{equation}
Combining \eqref{smeps1} and \eqref{smeps2}, we arrive at \eqref{estdifeps}.
\end{proof}

Lemma~\ref{lem:esidifeps} together with Proposition~\ref{prop:recqhH} yield Theorem~\ref{thm:uni} on the uniform stability of Inverse Problem~\ref{ip:simp} on the spectral data set $\mathring{B}_{\Omega, K}$.

\section{Unconditional uniform stability} \label{sec:uncond}

In this section, we consider certain classes of the non-self-adjoint Sturm-Liouville operators for which the condition $\| (I + \tilde R(x))^{-1} \|_{m \to m} \le K$ follows from some spectral data properties. As a result, we prove Theorem~\ref{thm:V}.

For $\Omega > 0$ and $\de \in (0, 1)$, consider the set $\mathring{\mathcal V}_{\Omega, \de} := \mathcal V_{\Omega, \de} \cap \mathring B_{\Omega}$, where $\mathcal V_{\Omega, \de}$ was defined by the requirements \eqref{reqrho} and \eqref{reqalpha}.
Each $S$ of $\mathring{\mathcal V}_{\Omega, \de}$ fulfills the Condition~1 of Theorem~\ref{thm:suff}, so the corresponding inverse operator $(I + \tilde R(x))^{-1}$ exists for each $x \in [0,\pi]$. Let us show that these operators are uniformly bounded with respect to $S \in \mathring{\mathcal V}_{\Omega, \de}$.

\begin{lem} \label{lem:V}
$\mathring{\mathcal V}_{\Omega, \de} \subset \mathring{B}_{\Omega, K}$ for some $K = C(\Omega, \de)$, that is, for each $S \in \mathring{\mathcal V}_{\Omega, \de}$ and $x \in [0,\pi]$, the inequality $\| (I + \tilde R(x))^{-1} \|_{m \to m} \le K$ holds.
\end{lem}

\begin{proof}
For a fixed $S \in \mathring{\mathcal V}_{\Omega, \de}$, we have 
\begin{equation} \label{estIRt}
\max_{x \in [0,\pi]} \| (I + \tilde R(x))^{-1} \|_{m \to m} < \infty,
\end{equation}
since $\tilde R(x)$ is continuous with respect to $x$. Let us prove the lemma by contradiction. Suppose that there exists a sequence $\{ S^{(k)} \} \subset \mathring{\mathcal V}_{\Omega, \de}$ such that
\begin{equation} \label{liminf}
\lim_{k \to \infty} \max_{x \in [0,\pi]} \| (I + \tilde R^{(k)}(x))^{-1} \|_{m \to m} = \infty.
\end{equation}

Put $\hat \rho_n^{(k)} := \rho_n^{(k)} - \tilde \rho_n$, $\hat \al_n^{(k)} := \al_n^{(k)} - \tilde \al_n$, and denote by $\hat S^{(k)}$ the sequence $\{ \hat \rho_1^{(k)}, \hat \al_1^{(k)}, \hat \rho_2^{(k)}, \hat \al_2^{(k)}, \dots \}$. It follows from \eqref{defBO} that $\hat S^{(k)} \in l_2^1$ and $\| \hat S^{(k)} \|_{l_2^1} \le \Omega$.
Then, from the sequence $\{ \hat S^{(k)} \}$, we can extract a subsequence such that $\{ \hat S^{(i_k)} \}$ weakly converges in $l_2^1$. Without loss of generality, put $\{ S^{(k)} \} = \{ S^{(i_k)} \}$. Denote the limit of $\{ S^{(k)} \}$ by $\hat S = \{ \hat \rho_1, \hat \al_1, \hat \rho_2, \hat \al_2, \dots \}$ and put $\rho_n := \tilde \rho_n + \hat \rho_n$, $\al_n := \tilde \al_n + \hat \al_n$, $S = \{ \rho_n, \al_n \}_{n \ge 1}$. Clearly, $S \in \mathring{B}_{\Omega}$. Furthermore, the weak convergence implies that $\{ S^{(k)} \}$ converges to $S$ element-wise, so $S$ fulfills the conditions \eqref{reqrho}.

Let us show that $S$ satisfies \eqref{reqalpha}. Since $S \in \mathring{B}_{\Omega}$, we have $|\hat \al_n| \le \frac{\Omega}{n}$, $n \ge 1$. Hence, starting from a sufficiently large index $N$, all $\alpha_n$'s are so close to $\tilde \al_n = \frac{2}{\pi}$ that $\arg \al_n \in (-\theta, \theta)$ for small $\theta > 0$, $n \ge N$. Next, choose such a number $k$ that the values $\{ \al_n^{(k)} \}_{n = 1}^{N-1}$ are so close to $\{ \al_n \}_{n = 1}^{N-1}$ that all $\{ \al_n \}_{n = 1}^{N-1}$ lie in some sector of angle $(\pi - (\de - \theta))$, because $\{ \al_n^{(k)} \}_{n = 1}^{N-1}$ lie in some sector of angle $(\pi - \de)$. Summarizing all above, we conclude that all $\{ \al_n \}_{n \ge 1}$ lie in a sector of angle $(\pi - (\de - \theta))$. Since $\theta > 0$ can be arbitrary, we arrive at \eqref{reqalpha}. 

Thus, $S \in \mathcal V_{\Omega, \de}$, so \eqref{estIRt} holds. Since $l_2^1$ is compactly embedded into $l_2$, then the weak convergence of $\{ \hat S^{(k)} \}$ in $l_2^1$ implies the strong convergence of $\{ \hat S^{(k)} \}$ to $S$ in $l_2$. In other words, the sequence $\{ \zeta_n^{(k)} \}$ tends to zero in $l_2$-norm as $k \to \infty$, where $\zeta_n^{(k)} := \sqrt{|\rho_n^{(k)} - \rho_n|^2 + |\al_n^{(k)} - \al_n|^2}$. By virtue of \eqref{estdifRt}, this implies that $\tilde R^{(k)}(x)$ tends to $\tilde R(x)$ in the operator norm uniformly with respect to $x \in [0,\pi]$. Consequently, the operator norms $\| (I + \tilde R^{(k)}(x))^{-1} \|$ are uniformly bounded with respect to $k$ and $x$. This contradicts to \eqref{liminf} and so concludes the proof.
\end{proof}

\begin{remark} \label{rem:Vpm}
The assertions similar to Lemma~\ref{lem:V} are valid for the sets $\mathring{\mathcal V}^{\pm}_{\Omega, \tau} := \mathcal V_{\Omega,\tau} \cap \mathring{B}_{\Omega}$ with a constant $K = C(\Omega, \tau)$. The proof is analogous. Indeed, the conditions $\al_n < 0$, $n \in \mathcal Z$, and $\pm \mbox{Im}\,\al_n \ge \tau_n$, $n \not\in \mathcal Z$, for the limit data $S$ follow from the element-wise convergence of the sequence $\{ S^{(k)} \}$.  
\end{remark}

Lemma~\ref{lem:V} and Remark~\ref{rem:Vpm} together with Theorem~\ref{thm:uni} imply the assertion of Theorem~\ref{thm:V} in the case $\om^{(1)} = \om^{(2)} = 0$. For passing to the general case, we apply a constant shift. Before that, we need the following remark.

\begin{remark} \label{rem:V}
The asymptotics \eqref{asympt} and $S \in \mathring{B}_{\Omega}$ imply that 
$$
|\rho_{n + 1} - \rho_n| \ge 1 - \frac{2 \Omega}{n}.
$$
Consequently, there exists a sufficiently large index $N = N(\Omega, \de)$ such that, for all $n \ge N$, the inequality $\rho_{n+1} - \rho_n \ge \de$ holds automatically if $\{ \rho_n \}$ are real. Therefore, we can equivalently replace the conditions on $\rho_n$ in \eqref{reqrho} by the requirements $\la_n \in \mathbb R$, $\la_1 \ge 0$, $\la_{n + 1} - \la_n \ge \de$ for $n \le N$. Moreover, the condition $\la_1 \ge 0$ here is just technical. It can be replaced by $\la_1 \ge -c$, $c > 0$. Then the constant $C$ in Theorem~\ref{thm:V} will additionally depend on $c$. We can completely remove this condition, since $S \in \mathring{B}_{\Omega}$ implies $\la_1 \ge -C(\Omega)$. To sum up, introduce the subsets $\mathring{\mathcal U}_{\Omega, \de}$ and $\mathring{\mathcal U}_{\Omega, \tau}^{\pm}$ of sequences $S \in \mathring{B}_{\Omega}$ satisfying the condition
$$
\la_n \in \mathbb R, \quad \la_{n + 1} - \la_n \ge \de, \: n \le N, \quad |\al_n| \ge \de, \: n \ge 1,
$$
and \eqref{reqalpha} or \eqref{reqalpha2}, respectively. Then, the assertion of Theorem~\ref{thm:V} is valid for $\mathring{\mathcal U}_{\Omega, \de}$ and $\mathring{\mathcal U}^{\pm}_{\Omega, \tau}$.
\end{remark}

\begin{proof}[Proof of Theorem~\ref{thm:V}]
Let $\{ \rho_n, \al_n \}_{n \ge 1}$ be the spectral data of a problem $L(q, h, H)$, $\la_n = \rho_n^2$. Obviously, a constant shift of the potential $\mathring{q} := q - c$ similarly shifts the eigenvalues $\mathring{\la}_n := \la_n - c$, $n \ge 1$, while the coefficients of the boundary conditions as well as the generalized weight numbers remain unchanged: $\mathring{h} = h$, $\mathring{H} = H$, $\mathring{\al}_n = \al_n$, $n \ge 1$. Applying the shift $c = \tfrac{2}{\pi} \om$, we get $\mathring{\om} = 0$.

Suppose that $S \in \mathcal V_{\Omega,\de}$ with $\om \ne 0$. Applying the shift, we get the new data $\mathring{S} = \{ \mathring{\rho}_n, \al_n \}_{n \ge 1} \in \mathring{\mathcal U}_{\Omega_0, \de_0}$, where $\Omega_0 > 0$ and $\de_0 \in (0, 1)$ depend only on $\Omega$ and $\de$. By shifting two spectral data sets $S^{(1)}$ and $S^{(2)}$ of $\mathcal V_{\Omega, \de}$ by $\tfrac{2}{\pi} \om^{(1)}$ and $\tfrac{2}{\pi} \om^{(2)}$, respectively ($\om^{(1)}$ and $\om^{(2)}$ can be distinct), we obtain the new sets $\mathring{S}^{(1)}$ and $\mathring{S}^{(2)}$ in $\mathring{\mathcal U}_{\Omega_0, \de_0}$. Remark~\ref{rem:V} implies the following stability estimate:
\begin{equation} \label{estsh}
\| \mathring{q}^{(1)} - \mathring{q}^{(2)} \|_{L_2} + |h^{(1)} - h^{(2)}| + |H^{(1)} - H^{(2)}| \le C(\Omega_0, \de_0) \sqrt{\sum_{n = 1}^{\infty} \bigl( |\mathring{\varkappa}_n^{(1)} - \mathring{\varkappa}_n^{(2)}|^2 + |s_n^{(1)} - s_n^{(2)}|^2 \bigr)}.
\end{equation}

One can easily show that
\begin{gather*}
\| \{ \mathring{\varkappa}_n^{(1)} - \mathring{\varkappa}_n^{(2)}\} \|_{l_2} \le C(\Omega) \| \{ \varkappa_n^{(1)} - \varkappa_n^{(2)}\} \|_{l_2}, \\
\| q^{(1)} - q^{(2)} \|_{L_2} \le \| \mathring{q}^{(1)} - \mathring{q}^{(2)} \|_{L_2} + \frac{2}{\sqrt \pi} |\om^{(1)} - \om^{(2)}|.
\end{gather*}
Combining the latter estimates with \eqref{estsh}, we arrive at \eqref{uni} with $C = C(\Omega, \de)$. Thus, the theorem is proved for $\mathcal V_{\Omega,\de}$. The proof for $\mathcal V_{\Omega,\tau}^{\pm}$ is similar.
\end{proof}

\section{Conditional uniform stability} \label{sec:cond}

In this section, we study the uniform stability of Inverse Problem~\ref{ip:simp} under a priori bounds on the parameters $q$, $h$, and $H$. In this case, we show that the operator $(I + \tilde R(x))^{-1}$ is uniformly bounded by using its explicit form \eqref{invR}. As a result, Theorem~\ref{thm:uniQA} on the unconditional uniform stability is deduced from Theorem~\ref{thm:uni}.

Let $Q > 0$ and $(q, h, H) \in P_Q$, where $P_Q$ was defined in \eqref{defPQ}. Obviously, $|\omega| \le C(Q)$, where $\om$ is defined by \eqref{defom}. The following lemma describes the spectral data properties for the problem $L(q, h, H)$.

\begin{lem} \label{lem:sdQ}
For $(q, h, H) \in P_Q$, there exists an index $N = N(Q) \ge 0$ such that:
\begin{enumerate}
\item For $n \le N$, the eigenvalues lie in the circle $|\rho_n| \le (N-\tfrac{3}{4})$.
\item For $n > N$, the eigenvalues are simple and lie in the circles $|\rho_n - \tilde \rho_n| \le \tfrac{1}{4}$.
\item The remainders in the asymptotics formulas \eqref{asympt} satisfy the uniform estimates
\begin{equation} \label{estrem}
\| \{ \varkappa_n\}_{n \ge 1} \|_{l_2} \le C(Q), \quad
\| \{ s_n \}_{n > N(Q)} \|_{l_2} \le C(Q).
\end{equation}
\end{enumerate}
\end{lem}

Location of eigenvalues according to Lemma~\ref{lem:sdQ} is presented in Fig.~\ref{img:eigen}, where stars and dots denote the points $\rho_n$ and $\tilde \rho_n$, respectively.

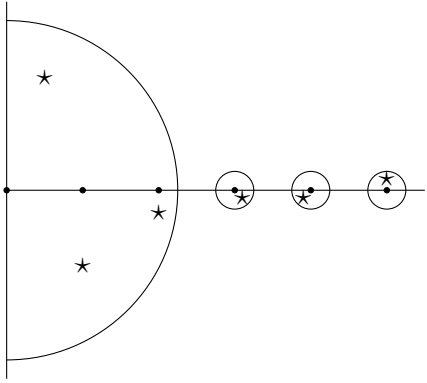
\begin{figure}[h!]
\centering
\begin{tikzpicture}
\draw (0,-2.25) arc(-90:90:2.25);
\draw (0,-2.5) edge (0,2.5);
\draw (0,0) edge (5.5,0);
\draw (3,0) circle (0.25);
\draw (4,0) circle (0.25);
\draw (5,0) circle (0.25);
\filldraw (0,0) circle (1pt);
\filldraw (1,0) circle (1pt);
\filldraw (2,0) circle (1pt);
\filldraw (3,0) circle (1pt);
\filldraw (4,0) circle (1pt);
\filldraw (5,0) circle (1pt);
\draw (0.5, 1.5) node{$\star$};
\draw (1, -1) node{$\star$};
\draw (2, -0.3) node{$\star$};
\draw (3.1, -0.1) node{$\star$};
\draw (3.9, -0.1) node{$\star$};
\draw (5, 0.15) node{$\star$};
\end{tikzpicture}
\caption{Location of eigenvalues}
\label{img:eigen}
\end{figure}

\begin{proof}[Proof of Lemma~\ref{lem:sdQ}]
The proof of this lemma is based on the standard relations \eqref{intDel} and \eqref{intDel0}, so we outline it briefly.
The square roots $\{ \rho_n \}_{n \ge 1}$ of the eigenvalues of the problem $L(q, h, H)$ coincide with the zeros of the characteristic function
$\Delta(\rho) = \vv'(\pi, \rho) + H \vv(\pi,\rho)$.
This function satisfies \eqref{intDel} with $\| \mathscr N \|_{L_2} \le C(Q)$.
Using \eqref{intDel} together with Rouche's Theorem, we get the desired properties of the eigenvalues.

The generalized weight numbers as the residues of the Weyl function $\al_n = \Res_{\la = \la_n} M(\la)$, where
\begin{equation} \label{fracM}
M(\rho^2) = -\frac{\Delta_0(\rho)}{\Delta(\rho)}
\end{equation}
(see \cite[Section 1.4.4]{FY01} for details).
In \eqref{fracM}, $\Delta_0(\rho)$ is the function defined in \eqref{defDelta}, which can be represented as \eqref{intDel0} with $\| \mathscr N_0 \|_{L_2} \le C(Q)$.
 
Due to the assertion 1 of this lemma, the poles of $M(\la)$ are simple for $n \ge N(Q)$ and so
\begin{equation} \label{alphapole}
\al_n = \frac{2 \rho_n \Delta^0(\rho_n)}{\dot \Delta(\rho_n)}, \quad n \ge N(Q).
\end{equation}

Substituting \eqref{intDel} and \eqref{intDel0} into \eqref{alphapole} and using \eqref{asympt}, we obtain \eqref{estrem} for $\{ s_n \}$.
\end{proof}

Using the representations \eqref{intDel}, \eqref{fracM}, \eqref{intDel0} and the estimates $\| \mathscr N \|_{L_2}, \, \| \mathscr N_0 \|_{L_2} \le C(Q)$, we obtain the following corollary about the Weyl function.

\begin{cor} \label{cor:estM}
$|M(\rho^2)| \le C(Q)$ for $|\rho| = N(Q) - \tfrac{1}{2}$.
\end{cor}

Note that the eigenvalues for $n \le N(Q)$ are not necessarily simple. Even if they are simple, the corresponding generalized weight numbers $\{ \al_n \}_{n \le N(Q)}$ are not uniformly bounded for $(q, h, H) \in P_Q$, which is shown by the following example.

\begin{example}[\hspace*{-3pt}\cite{Bond20}] \label{ex:unbound}
Let $ L = L(q, h, H)$ be a non-self-adjoint Sturm-Liouville problem with a double eigenvalue $\la_1 = \la_2$ and the Weyl function
$$
M(\la) \sim \frac{\al_2}{(\la - \la_1)^2} + \frac{\al_1}{\la - \la_1}
$$
in the neighbourhood of $\la_1$. For sufficiently small $\de > 0$, consider perturbed spectral data
\begin{gather*}
\la_1^{(\de)} := \la_1 + \de, \quad \la_2^{(\de)} := \la_1 - \de + c \de^2, \quad
\al_1^{(\de)} := \frac{a}{\de}, \quad \al_2^{(\de)} := -\frac{a}{\de}, \\
\la_n^{(\de)} := \la_n, \quad \al_n^{(\de)} := \al_n, \quad n \ge 3,
\end{gather*}
where $a := \frac{\al_2}{2}$, $c := \frac{\al_1}{a}$. According to Theorem~2.3 in \cite{Bond20}, for sufficiently small $\de > 0$, the solution $(q^{(\de)}, h^{(\de)}, H^{(\de)})$ for the inverse problem by the spectral data $\{ \la_n^{(\de)}, \al_n^{(\de)} \}_{n \ge 1}$ exists and fulfills the local stability estimate
$$
\| q^{(\de)} - q \|_{L_2} + |h^{(\de)} - h| + |H^{(\de)} - H| \le C(q,h,H) \de^2.
$$
Thus, the triples $(q^{(\de)}, h^{(\de)}, H^{(\de)})$ are bounded but the generalized weight numbers $\al_1^{(\de)}$ and $\al_2^{(\de)}$ are unbounded as $\de \to 0$. Such numerical examples are presented in \cite{BG21} for the quadratic Sturm-Liouville pencils.
\end{example}

In order to exclude the cases similar to Example~\ref{ex:unbound}, we impose the restriction $|\al_n| \le A$ for some $A > 0$ and consider the class $P_{Q, A}$ defined in Section~\ref{sec:main}. Put 
$$
\mathring{P}_{Q, A} := \{ (q,h,H) \in P_{Q, A} \colon \om = 0 \}.
$$

Consider $(q, h, H) \in \mathring{P}_{Q, A}$. Using the problem $L(q, h, H)$ and the model problem $\tilde L$, construct the operator $R(x)$ via \eqref{defD}, \eqref{defR1}, \eqref{defR2}, and \eqref{operR}.

\begin{lem} \label{lem:RQA}
For $(q, h, H) \in \mathring{P}_{Q, A}$, the following estimate is valid:
\begin{equation*}
\| R(x) \|_{m \to m} \le C(Q, A), \quad x \in [0,\pi].
\end{equation*}
\end{lem}

\begin{proof}
The solution $\vv(x, \rho)$ satisfies the Volterra integral equation
$$
\vv(x, \rho) = \cos \rho x + h \frac{\sin \rho x}{\rho} + \frac{1}{\rho}\int_0^x \sin \rho(x - t) q(t) \vv(t, \rho) \, dt.
$$
Hence
\begin{equation} \label{estvvQ}
|\vv(x, \rho)| \le C(Q) \exp(|\mbox{Im}\,\rho|), \quad 
|\vv'(x, \rho)| \le C(Q) |\rho| \exp(|\mbox{Im} \, \rho|).
\end{equation}
In view of \eqref{eqv}, the function $D(x, \rho, \theta)$ defined by \eqref{defD} can be represented as follows:
$$
D(x, \rho, \theta) = \frac{\vv(x, \rho) \vv'(x, \theta) - \vv'(x, \rho) \vv(x, \theta)}{\rho^2 - \theta^2}.
$$
Relying \eqref{estvvQ}, we obtain the estimate
\begin{gather} \label{estD}
|D(x, \rho, \theta)| \le \frac{C(Q)}{|\rho - \theta| + 1}, \qquad
\text{for} \:\: \mbox{Re}\,\rho, \, \mbox{Re}\,\theta \ge 0, \quad |\mbox{Im}\,\rho|, |\mbox{Im}\,\theta| \le C(Q).
\end{gather}

Using \eqref{not}, \eqref{estD}, Lemma~\ref{lem:sdQ}, and Schwarz's Lemma, we get
\begin{gather*}
|D(x, \rho_{ni}, \rho_{kj})| \le \frac{C}{|n - k| + 1}, \quad
|D(x, \rho_{n0}, \rho_{kj}) - D(x, \rho_{n1}, \rho_{kj})| \le \frac{C |\hat \rho_n|}{|n-k| + 1}, \\
|D(x, \rho_{ni}, \rho_{k0}) - D(x, \rho_{ni}, \rho_{k1})| \le \frac{C|\hat \rho_k|}{|n -k| + 1}, \\
|D(x, \rho_{n0}, \rho_{k0}) - D(x, \rho_{n1}, \rho_{k0}) - D(x, \rho_{n0}, \rho_{k1}) + D(x, \rho_{n1}, \rho_{k1})| \le \frac{C |\hat \rho_n| |\hat \rho_k|}{|n - k| + 1},
\end{gather*}
where $C = C(Q)$, $n, k \ge 1$, $i,j = 0, 1$, $x \in [0, \pi]$.
The analogous estimates are valid for $\dot D$. Using these estimates together with $|\al_n| \le A$ and the definitions \eqref{defR1}--\eqref{defR2}, we deduce
\begin{equation} \label{estRQA}
|R_{ni,kj}(x)| \le \frac{C(Q, A) \xi_k}{|n - k| + 1}.
\end{equation}

In view of Lemma~\ref{lem:sdQ}, we get $\| \{ \xi_k \} \|_{l_2} \le C(Q, A)$.
Using \eqref{estopR} and \eqref{estRQA}, we obtain
$$
    \| R(x) \|_{m \to m} \le C(Q, A) \| \{ \xi_k \} \|_{l_2} \le C(Q, A),
$$
which yields the claim.
\end{proof}

Lemma~\ref{lem:RQA} together with \eqref{invR} imply
$$
\| (I + \tilde R(x))^{-1} \|_{m \to m} \le C(Q, A) \quad \text{for} \:\: (q, h, H) \in \mathring{P}_{Q, A}.
$$
Furthermore, it follows from Lemma~\ref{lem:sdQ} and $|\al_n| \le A$ that $\Omega \le C(Q, A)$. These arguments lead to the following lemma.

\begin{lem} \label{lem:mainQA}
For $(q, h, H) \in \mathring{P}_{Q, A}$, the spectral data $S$ belong to $\mathring{B}_{\Omega, K}$, where $\Omega = C(Q, A)$ and $K = C(Q, A)$.
\end{lem}

Combining Lemma~\ref{lem:mainQA} with Theorem~\ref{thm:uni}, we arrive at the assertion of Theorem~\ref{thm:uniQA} for $\om^{(1)} = \om^{(2)} = 0$. One can easily pass to the general case by a shift $\mathring{q} := q - \tfrac{2}{\pi} \om$, $\mathring{\la}_n := \la_n - \tfrac{2}{\pi} \om$ as in the proof of Theorem~\ref{thm:V} in Section~\ref{sec:cond}. This concludes the proof of Theorem~\ref{thm:uniQA} on the conditional uniform stability of the inverse problem.

\section{Multiple eigenvalues} \label{sec:mult}

This section deals with the case when the eigenvalues of the non-self-adjoint Sturm-Liouville problem \eqref{eqv}--\eqref{bc} are not necessarily simple. Throughout this section, we mean that $\{ \al_n \}_{n \ge 1}$ are the coefficients of the expansion \eqref{Laurent} of the Weyl function and study Inverse Problem~\ref{ip:simp}. Specifically, we obtain the generalization of the modified main equation \eqref{main} to the case of multiple eigenvalues by developing the approach of \cite{Bond20}. A feature of our stability analysis is that the spectral data $S^{(1)}$ and $S^{(2)}$ of two different problems can have different multiplicities for some finite number of eigenvalues. Therefore, in the main equation and other relations, we separately form an integral part over the contour $\Gamma_N$ defined by \eqref{defGa}.
As a result, we generalize Theorem~\ref{thm:uni} to the case of multiple eigenvalues and prove Theorem~\ref{thm:uniQ} about the conditional uniform stability.

Fix $N \ge 0$ and introduce the set
$$
\mathcal S_N := \left\{ \{ \rho_n, \al_n \}_{n \ge 1} \in \mathcal S \colon \begin{array}{ll}
|\rho_n| \le N-\tfrac{3}{4}, & n \le N \\
|\rho_n| \ge N - \tfrac{1}{4}, \: \rho_n \ne \rho_k, & n,k > N 
\end{array}
\right\}.
$$

Clearly, $\mathcal S_0$ is the subset of $\mathcal S$ with simple eigenvalues, $\mathcal S = \bigcup_{N = 0}^{\infty} \mathcal S_N$ in view of the eigenvalue asymptotics \eqref{asympt}, and the spectral data $\{ \tilde \rho_n, \tilde \al_n \}_{n \ge 1}$ of $\tilde L$ belong to $\mathcal S_N$ for any $N \ge 0$. Consequently, for any $S \in \mathcal S_N$, the function $\hat M(\rho^2) := M(\rho^2) - \tilde M(\rho^2)$ is continuous on the contour 
$\Gamma_N$ defined by \eqref{defGa}.

Let us construct the main equation for the case of multiple eigenvalues. The relation \eqref{relvv} can be generalized as follows (see \cite{Bond20}):
\begin{align} \nonumber
\tilde \vv(x, \rho) = & \, \vv(x, \rho) + \frac{1}{2 \pi i} \oint_{\Gamma_N} \theta \hat M(\theta^2) \tilde D(x, \rho, \theta) \vv(x, \theta) \, d\theta \\ \label{relvv1} & + \sum_{k > N} \bigl(\al_{k0} \tilde D(x, \rho, \rho_{k0}) \vv_{k0}(x) - \al_{k1} \tilde D(x, \rho, \rho_{k1}) \vv_{k1}(x) \bigr).
\end{align}

Using \eqref{relvv1}, we deduce a linear equation in a special Banach space having continuous and discrete part.

Denote by $\mathfrak B_C = \mathfrak B_{N, C}$ the Banach space of continuous functions on $\Gamma_N$ with the norm
$$
\| f_C \|_{\mathfrak B_C} = \max_{\rho \in \Gamma_N} |f_C(\rho)|, \quad f_C \in \mathfrak B_C.
$$

Denote by $\mathfrak B_D = \mathfrak B_{N, D}$ the Banach space of bounded infinite sequences $f_D = [f_{ni}]_{n \ge N + 1, i = 0, 1}$ with the norm
$$
\| f_D \|_{\mathfrak B_D} = \sup_{n,i} |f_{ni}|, \quad f_D \in \mathfrak B_D.
$$

Here and below, the indices $C$ and $D$ mean ``continuous'' and ``discrete''. Although the spaces $\mathfrak B_C$ and $\mathfrak B_D$ depend on $N$, we assume the index $N$ to be fixed and omit it for brevity.

Define the Banach space
$$
\mathfrak B = \mathfrak B_N := \{ f = [f_C, f_D] \colon f_C \in \mathfrak B_C, f_D \in \mathfrak B_D \}, \quad \| f \|_{\mathfrak B} := \| f_C \|_{\mathfrak B_C} + \| f_D \|_{\mathfrak B_D}.
$$

For each $x \in [0,\pi]$, define the element $\psi(x) = [\psi_C(x), \psi_D(x)] \in \mathfrak B$ as follows:
$$
\psi_C(x, \rho) := \vv(x, \rho), \: \rho \in \Gamma_N, \quad 
\psi_D(x) = [\psi_{ni}(x)]_{n > N, \, i = 0, 1}, 
$$
where $\psi_{ni}(x)$ were defined in \eqref{defpsi}. Furthermore, introduce the linear bounded operator $R(x) \colon \mathfrak B \to \mathfrak B$ as follows:
\begin{align} \label{defRmult}
& R(x) = \begin{bmatrix}
            R_{CC}(x) & R_{CD}(x) \\
            R_{DC}(x) & R_{DD}(x)
      \end{bmatrix}, \quad
\begin{array}{ll}
R_{CC}(x) \colon \mathfrak B_C \to \mathfrak B_C, & R_{CD}(x) \colon \mathfrak B_C \to \mathfrak B_D, \\
R_{DC}(x) \colon \mathfrak B_D \to \mathfrak B_C, & R_{DD}(x) \colon \mathfrak B_D \to \mathfrak B_D,
\end{array} \\ \nonumber
& R(x) f = \bigl[ R_{CC}(x) f_C + R_{CD}(x) f_D, R_{DC}(x) f_C + R_{DD}(x) f_D \bigr], \quad f = [f_C, f_D] \in \mathfrak B, \\ \label{defRCC}
& (R_{CC}(x) f_C)(\rho) = r(x, \rho) := \frac{1}{2 \pi i} \oint_{\Gamma_N} \theta \hat M(\mu) D(x, \rho, \theta) f_C(\theta) \, d\theta,  \\ \nonumber
& (R_{CD}(x) f_D)(\rho) = \sum_{k > N} \bigl[ \al_{k0}  D(x, \rho, \rho_{k0}), \, -\al_{k1} D(x, \rho, \rho_{k1}) \bigr] T_k \begin{bmatrix} f_{k0} \\ f_{k1} \end{bmatrix}.
\end{align}
Next, the sequence $z_{ni}(x) := (R_{DC}(x) f_D)_{ni}$, $n > N$, $i = 0, 1$, is defined analogously to \eqref{defpsi}, where $r(x, \rho)$ was defined in \eqref{defRCC}:
$$
\begin{bmatrix}
z_{n0}(x) \\
z_{n1}(x)
\end{bmatrix} = T_n^{-1} \begin{bmatrix}
r(x, \rho_{n0}) \\
r(x, \rho_{n1})
\end{bmatrix}, \: \rho_{n0} \ne \rho_{n1}, \quad
\begin{bmatrix}
z_{n0}(x) \\
z_{n1}(x)
\end{bmatrix} = 
\begin{bmatrix}
\dot r(x, \rho_{n1}) \\
r(x, \rho_{n1})
\end{bmatrix}, \: \rho_{n0} = \rho_{n1}.
$$
The operator $R_{DD}(x)$ is defined analogously to \eqref{operR} by using the elements $R_{ni,kj}(x)$ given by \eqref{defR1} and \eqref{defR2}:
$$
(R_{DD}(x) f_D)_{ni} = \sum_{k > N, \, j = 0, 1} R_{ni,kj}(x) f_{kj}, \quad n > N, \, i = 0, 1.
$$
Similarly to $\psi(x)$ and $R(x)$, we define $\tilde \psi(x)$ and $\tilde R(x)$. 

In terms of the new definitions, the modified main equation \eqref{main} holds as a linear equation in $\mathfrak B$ for each fixed $x \in [0,\pi]$. Now, $I$ denotes the identity operator in $\mathfrak B$. Similarly to Theorem~\ref{thm:nsc}, we get the following necessary and sufficient conditions for solvability of Inverse Problem~\ref{ip:simp}.

\begin{thm}
For complex numbers $\{ \rho_n, \al_n \}_{n \ge 1} \in \mathcal S_N$ to be the spectral data of a problem $L = L(q, h, H)$, the existence of a bounded inverse operator $(I + \tilde R(x))^{-1}$ in $\mathfrak B$ for each fixed $x \in [0,\pi]$ is necessary and sufficient.
\end{thm}

The solution of Inverse Problem~\ref{ip:simp} can be found by formulas \eqref{recqhH}, where
\begin{equation} \label{defepsint}
\eps(x) := \frac{1}{2 \pi i} \oint_{\Gamma_N} \theta \hat M(\theta) \psi(x, \theta) \tilde \psi(x, \theta) \, d\theta + \sum_{k = N + 1}^{\infty} \bigl(\al_{k0} \vv_{k0}(x) \tilde \vv_{k0}(x) - \al_{k1} \vv_{k1}(x) \tilde \vv_{k1}(x)\bigr),
\end{equation}
and $\vv_{kj}(x)$ are determined by \eqref{vvpsi}.

Proceed to the uniform stability of Inverse Problem~\ref{ip:simp} with multiple eigenvalues.
For $N \ge 0$, $\Omega > 0$, and $K > 0$, define the following analog of $\mathring{B}_{\Omega, K}$:
$$
\mathring{B}_{N, \Omega, K} := \bigl\{ S \in \mathcal S_N \colon d_N(S, \tilde S) \le \Omega, \, \om = 0, \, \| (I + \tilde R(x))^{-1} \|_{\mathfrak B \to \mathfrak B} \le K \bigr\}. 
$$

The following theorem on the uniform stability is obtained similarly to Theorem~\ref{thm:uni}.

\begin{thm} \label{thm:unimult}
Suppose that $S^{(1)}, S^{(2)} \in \mathring{B}_{N, \Omega, K}$. Then, the corresponding solutions $(q^{(1)}, h^{(1)}, H^{(1)})$ and $(q^{(2)}, h^{(2)}, H^{(2)})$ of the inverse problem satisfy the uniform estimate \eqref{unimult} holds with $C = C(N, \Omega, K)$.
\end{thm}

The proof Theorem~\ref{thm:unimult} is based on the estimates for $\tilde \psi^{(1)} - \tilde \psi^{(2)}$, $\tilde R^{(1)} - \tilde R^{(2)}$, $R^{(1)} - R^{(2)}$, etc. However, for continuous parts, obtaining such estimates is easy, since the contour $\Gamma_N$ is bounded. For discrete parts, the arguments are exactly the same as in Sections~\ref{sec:ubound} and~\ref{sec:ustab}. So, we omit the proof of Theorem~\ref{thm:unimult}.

Theorem~\ref{thm:unimult} implies Theorem~\ref{thm:uniQ} on the conditional uniform stability for the inverse problem with multiple eigenvalues.

\begin{proof}[Proof of Theorem~\ref{thm:uniQ}]
Consider $(q, h, H) \in \mathring{P}_Q$. By virtue of Lemma~\ref{lem:sdQ}, there exists $N = N(Q)$ such that the spectral data $S$ of $L(q, h, H)$ belong to $\mathcal S_N$. Construct the operators $R(x)$ and $\tilde R(x)$ in the Banach space $\mathfrak B = \mathfrak B_N$ by \eqref{defRmult} and subsequent formulas.
Similarly to Lemma~\ref{lem:RQA}, we prove that 
$$
\| R(x) \|_{\mathfrak B \to \mathfrak B} \le C(Q), \quad x \in [0,\pi]. 
$$
At the same time, the relation analogous to \eqref{invR} holds:
$$
(I + \tilde R(x))^{-1} = I - R(x).
$$
Hence 
\begin{equation} \label{estRtmult}
\| (I + \tilde R(x))^{-1} \|_{\mathfrak B \to \mathfrak B} \le C(Q), \quad x \in [0,\pi], \quad (q, h, H) \in \mathring{P}_Q.
\end{equation}

Next, Lemma~\ref{lem:sdQ} and Corollary~\ref{cor:estM} directly imply that
\begin{equation} \label{estCQ1}
\sqrt{\sum_{k > N} |\varkappa_n - \tilde \varkappa_n|^2 + |s_n - \tilde s_n|^2} \le C(Q)
\end{equation}
and
\begin{equation} \label{estCQ2}
\max_{\rho \in \Gamma_N} |M(\rho^2) - \tilde M(\rho^2)| \le C(Q),
\end{equation}
respectively. Combining \eqref{defdN}, \eqref{estCQ1}, and \eqref{estCQ2}, we get 
\begin{equation} \label{estdN}
|d_N(S, \tilde S)| \le C(Q), \quad (q, h, H) \in \mathring{P}_Q.
\end{equation}

Using Lemma~\ref{lem:sdQ}, \eqref{estRtmult}, and \eqref{estdN}, we conclude that spectral data $S$ corresponding to any $(q, h, H) \in \mathring{P}_Q$ belong to $\mathring{B}_{N, \Omega, K}$, where $N$, $\Omega$, and $K$ depend only on $Q$. Applying Theorem~\ref{thm:unimult}, we obtain the estimate \eqref{unimult} with $C = C(Q)$ and $N = N(Q)$ for the case $\om^{(1)} = \om^{(2)} = 0$. Finally, we pass to the general case by a shift, which concludes the proof.
\end{proof}

\begin{remark}
Theorems~\ref{thm:unimult} and \ref{thm:uniQ}, as well as our construction of the modified main equation remain valid if $M(\la)$ in the definition of $d_N$ and under the contour integrals is replaced by the rational function $M_N(\la)$, which is constructed by the finite part of the spectral data $\{ \la_n, \al_n \}_{n \le N}$:
$$
M_N(\la) := \sum_{n \in \mathcal I_N} \sum_{\nu = 0}^{m_n - 1} \frac{\al_{n + \nu}}{(\la- \la_n)^{\nu + 1}}, \quad
\mathcal I_N := \mathcal I \cap \{ 1, 2, \dots, N \}.
$$
Indeed, the remainders $\sum_{k > N} \frac{\al_k}{\la - \la_k}$ and $\sum_{k > N} \frac{\tilde \al_k}{\la - \tilde \la_k}$ do not influence on the contour integrals in \eqref{relvv1}, \eqref{defRCC}, and \eqref{defepsint}. Furthermore, the asymptotics \eqref{asympt} and the estimates \eqref{estCQ1} imply the estimate similar to \eqref{estCQ2} for the remainders. Therefore, the analogous estimate holds for $|M_N^{(1)}(\rho^2) - M_N^{(2)}(\rho^2)|$. Thus, $M^{(1)} - M^{(2)}$ can be replaced by $M_N^{(1)} - M_N^{(2)}$ in $d_N$ in Theorem~\ref{thm:uniQ}. Then, the data $\{ \rho_n, \al_n \}$ for $n \le N$ and for $n > N$ will participate in $d_N$ separately.
However, the result with $M^{(1)}-M^{(2)}$ on the contour $\Gamma_N$ is more convenient for the application in the next section.
\end{remark}

\section{Inverse problem by Cauchy data} \label{sec:Cauchy}

In this section, we study Inverse Problem~\ref{ip:Cauchy} by the Cauchy data $\mathscr C := \{ \mathscr N, \mathscr N_0, \om, \om_0 \}$ and prove Theorem~\ref{thm:uniCauchy} on the uniform stability.

The Cauchy data $\mathscr C$ can be treated as an element of the linear space $\mathcal C := L_2(0,\pi) \times L_2(0,\pi) \times \mathbb C \times \mathbb C$ with the norm
$$
\| \mathscr C \|_{\mathcal C} := \| \mathscr N \|_{L_2} + \| \mathscr N_0 \|_{L_2} + |\om| + |\om_0|.
$$

Let $\mathscr C = \{ \mathscr N, \mathscr N_0, \om, \om_0 \}$ be an arbitrary element of $\mathcal C$, not necessarily related to some problem $L(q, h, H)$. Define the functions $\Delta(\rho)$ and $\Delta_0(\rho)$ by formulas \eqref{intDel} and \eqref{intDel0}, respectively. Denote by $\{ \pm \rho_n \}_{n \ge 1}$ the zeros of $\Delta(\rho)$ (counting with multiplicities) so that $\arg \rho_n \in \left[ -\frac{\pi}{2}, \tfrac{\pi}{2}\right)$. Put $M(\rho^2) := -\dfrac{\Delta_0(\rho)}{\Delta(\rho)}$ and denote by $\{ \al_n \}_{n \ge 1}$ the coefficients of the expansion \eqref{Laurent} for $M(\la)$. Thus, we have defined the ``spectral data'' $\{ \rho_n, \al_n \}_{n \ge 1} \in \mathcal S$ associated with $\mathscr C \in \mathcal C$.

\begin{lem} \label{lem:CSD}
For every $W > 0$, there exists $N = N(W)$ such that, for any
$\mathscr C^{(s)} \in \mathcal C$ satisfying $\| \mathscr C^{(s)} \|_{\mathcal C} \le W$ for $s = 1, 2$. the corresponding data $S^{(s)} = \{ \rho_n^{(s)}, \al_n^{(s)} \}_{n \ge 1}$, $s = 1, 2$, fulfill the uniform estimate
\begin{equation} \label{estCSD}
d_N(S^{(1)}, S^{(2)}) \le C(W) \| \mathscr C^{(1)} - \mathscr C^{(2)} \|_{\mathcal C}.
\end{equation}
Thus, the mapping $\mathscr C \mapsto S$ is Lipschitz continuous on the ball $\| \mathscr S \|_{\mathcal C} \le W$. 
\end{lem} 

For proving Lemma~\ref{lem:CSD}, it is convenient first to apply the shift that removes $\omega$. So, we need the following auxiliary lemma.

\begin{lem} \label{lem:auxC}
For $s = 1, 2$, let 
\begin{equation} \label{intDelNs}
\Delta^{(s)}(\rho) := -\rho \sin \rho \pi + \om^{(s)} \cos \rho \pi + \int_0^{\pi} \mathscr N^{(s)}(t) \cos \rho t \, dt,
\end{equation}
where $\mathscr N^{(s)} \in L_2(0,\pi)$ and $\om^{(s)} \in \mathbb C$ satisfy $\| \mathscr N^{(s)} \|_{L_2} + |\om^{(s)}| \le W$. Then
\begin{equation} \label{intDelMs}
\Delta^{(s)}\left(\sqrt{\rho^2 + \tfrac{2}{\pi} \om^{(s)}}\right) = -\rho \sin \rho \pi + \int_0^{\pi} \mathscr M^{(s)}(t) \cos \rho t \, dt,
\end{equation}
where $\mathscr M^{(s)} \in L_2(0,\pi)$, $\| \mathscr M^{(s)} \|_{L_2} \le C(W)$, and
\begin{equation} \label{estdifM}
\| \mathscr M^{(1)} - \mathscr M^{(2)} \|_{L_2} \le C(W) \bigl( \| \mathscr N^{(1)} - \mathscr N^{(2)} \|_{L_2} + |\om^{(1)} - \om^{(2)}| \bigr).
\end{equation}
Thus, the mapping $(\mathscr N, \om) \mapsto \mathscr M$ is Lipschitz continuous on the ball $\| \mathscr N \|_{L_2} + |\om| \le W$.
\end{lem}

\begin{proof}
Using \eqref{intDelNs}, one can easily show that $f(\rho) := \Delta^{(s)}(\sqrt{\rho^2 + \tfrac{2}{\pi} \om^{(s)}}) + \rho \sin \rho \pi$ is an even Paley-Wiener function with the norm $\| f \|_{L_2(\mathbb R)} \le C(W)$. This readily implies the representation \eqref{intDelMs} and the estimate $\| \mathscr M^{(s)} \|_{L_2} \le C(W)$.

Proceed to obtaining the estimate \eqref{estdifM}. Define the cosine Fourier coefficients
$$
\hat{\mathscr M}^{(s)}_n := \int_0^{\pi} \mathscr M^{(s)}(t) \cos n t \, dt, \quad n \ge 0, \quad s = 1, 2.
$$
By virtue of \eqref{intDelMs}, we get
$\hat{\mathscr M}^{(s)}_n = \Delta^{(s)}(\rho_n(\om^{(s)}))$, where
\begin{equation} \label{rhonom}
\rho_n(\om) := \sqrt{n + \frac{2}{\pi}\om} = n + \frac{\om}{\pi n} \left(1 + O(n^{-2}) \right).
\end{equation}
Denote the main part of \eqref{intDelNs} as 
$$
p_n(\om) := -\rho_n(\om) \sin \rho_n(\om)\pi + \om \cos \rho_n(\om)\pi. 
$$
Using \eqref{rhonom}, we obtain
\begin{equation} \label{estpn}
\bigl| p_n(\om^{(1)}) - p_n(\om^{(2)})\bigr| \le \frac{C(W) |\om^{(1)} - \om^{(2)}|}{(n+1)^2}, \quad n \ge 0.
\end{equation}
Next, we have
\begin{multline} \label{difN12}
\int_0^{\pi} \mathscr N^{(1)}(t) \cos \rho_n(\om^{(1)}) t \, dt - 
\int_0^{\pi} \mathscr N^{(2)}(t) \cos \rho_n(\om^{(2)}) t \, dt \\ =
\int_0^{\pi} (\mathscr N^{(1)} - \mathscr N^{(2)})(t) \cos \rho_n(\om^{(1)}) t \, dt + \int_0^{\pi} \mathscr N^{(2)}(t) \bigl(\cos \rho_n(\om^{(1)}) t - \cos \rho_n(\om^{(2)}) t\bigr) \, dt
\end{multline}
Calculations show that
\begin{multline} \label{estdifN1}
\left| \int_0^{\pi} (\mathscr N^{(1)} - \mathscr N^{(2)})(t) \cos \rho_n(\om^{(1)}) t \, dt \right| \le |\hat{\mathscr N}_n^{(1)} - \hat{\mathscr N}_n^{(2)}| \\ + C(W)\left( \frac{|\hat{\mathscr S}_n^{(1)} - \hat{\mathscr S}_n^{(2)}|}{n + 1} + \frac{\| \mathscr N^{(1)} - \mathscr N^{(2)}\|_{L_2}}{(n + 1)^2} \right),
\end{multline}
\begin{equation} \label{estdifN2}
\left| \int_0^{\pi} \mathscr N^{(2)}(t) \bigl(\cos \rho_n(\om^{(1)}) t - \cos \rho_n(\om^{(2)}) t\bigr) \, dt \right| \le \frac{|\om^{(1)} - \om^{(2)}| w_n}{n + 1},
\end{equation}
where
$$
\hat{\mathscr N}^{(s)}_n := \int_0^{\pi} \mathscr N^{(s)}(t) \cos n t \, dt, \quad
\hat{\mathscr S}^{(s)}_n := \int_0^{\pi} \mathscr N^{(s)}(t) \, t \sin n t \, dt, \quad \| \{ w_n \} \|_{l_2} \le C(W).
$$

Combining \eqref{intDelNs}, \eqref{estpn}, \eqref{difN12}, \eqref{estdifN1}, and \eqref{estdifN2}, we get
$$
\bigl| \hat{\mathscr M}_n^{(1)} - \hat{\mathscr M}_n^{(2)} \bigr| \le |\hat{\mathscr N}_n^{(1)} - \hat{\mathscr N}_n^{(2)}| + C(W)\left( \frac{|\hat{\mathscr S}_n^{(1)} - \hat{\mathscr S}_n^{(2)}|}{n + 1} + \frac{\| \mathscr N^{(1)} - \mathscr N^{(2)}\|_{L_2}}{(n + 1)^2} \right) + \frac{|\om^{(1)} - \om^{(2)}| w_n}{n + 1}.
$$
Applying Bessel's inequality, we obtain
$$
\left( \sum_{n = 0}^{\infty} \bigl| \hat{\mathscr M}_n^{(1)} - \hat{\mathscr M}_n^{(2)} \bigr|^2 \right)^{1/2} \le C(W) \bigl( \| \mathscr N^{(1)} - \mathscr N^{(2)} \|_{L_2} + |\om^{(1)} - \om^{(2)}|\bigr),
$$
which implies \eqref{estdifM}.
\end{proof}

\begin{proof}[Proof of Lemma~\ref{lem:CSD}]
Let $W > 0$ be fixed, collections $\mathscr C^{(s)}$ satisfy the hypothesis of Lemma~\ref{lem:CSD}, and $\om^{(s)} = 0$, $s = 1, 2$. Construct the functions $\Delta^{(s)}(\rho)$ by \eqref{intDelNs} and
\begin{equation} \label{intDel0Ns}
\Delta_0^{(s)}(\rho) := \cos \rho \pi + \om_0^{(s)} \frac{\sin \rho \pi}{\rho} + \frac{1}{\rho} \int_0^{\pi} \mathscr N^{(s)}_0(t) \sin \rho t \, dt, \quad s = 1, 2.
\end{equation}
Since $\| \mathscr C^{(s)} \|_{\mathcal C} \le W$, then the relations \eqref{intDelNs} and \eqref{intDel0Ns} imply the following estimates for sufficiently large $N = N(W)$, $\rho \in \Gamma_N$:
\begin{gather*}
0 < c(W) \le |\Delta^{(s)}(\rho)| \le C(W), \quad |\Delta_0^{(s)}(\rho)| \le C(W), \quad s = 1, 2, \\
\bigl|\Delta^{(1)}(\rho) - \Delta^{(2)}(\rho)\bigr| \le C(W) \| \mathscr N^{(1)} - \mathscr N^{(2)} \|_{L_2}, \\
\bigl|\Delta_0^{(1)}(\rho) - \Delta_0^{(2)}(\rho)\bigr| \le C(W) \bigl( \| \mathscr N_0^{(1)} - \mathscr N_0^{(2)} \|_{L_2} + |\om^{(1)} - \om^{(2)}|\bigr).
\end{gather*}
Then
\begin{multline} \label{CSDM}
\bigl| M^{(1)}(\rho^2) - M^{(2)}(\rho^2)\bigr| = \left| \frac{\Delta_0^{(1)}(\rho)}{\Delta^{(1)}(\rho)} - \frac{\Delta_0^{(2)}(\rho)}{\Delta^{(2)}(\rho)}\right| \\ \le
\frac{|(\Delta^{(1)}_0 - \Delta^{(2)}_0)(\rho)| |\Delta^{(2)}(\rho)| + |\Delta^{(2)}_0(\rho)| |(\Delta^{(1)} - \Delta^{(2)})(\rho)|}{|\Delta^{(1)}(\rho)| |\Delta^{(2)}(\rho)|} \le C(W) \| \mathscr C^{(1)} - \mathscr C^{(2)} \|_{\mathcal C}, \quad \rho \in \Gamma_N.
\end{multline}

Next, one can easily see from the representation
\begin{equation} \label{Deltashort}
\Delta^{(s)}(\rho) = -\rho \sin \rho \pi + \int_0^{\pi} \mathscr N^{(s)}(t) \cos \rho t \, dt
\end{equation}
that the zeros $\{ \pm \rho_n^{(s)} \}_{n \ge 1}$ of $\Delta^{(s)}(\rho)$ satisfy the asymptotics
$$
\rho_n^{(s)} = n - 1 + \frac{\varkappa_n^{(s)}}{n}, \quad \{ \varkappa_n^{(s)} \} \in l_2, \quad n \ge 1, \quad s = 1, 2.
$$
Furthermore, the index $N = N(W)$ can be chosen so large that $\{ \rho_n^{(s)} \}_{n > N}$ are simple and
\begin{equation} \label{CSDrho}
\left( \sum_{n = N + 1}^{\infty} \bigl| \varkappa_n^{(1)} - \varkappa_n^{(2)} \bigr|^2 \right)^{1/2} \le C(W) \| \mathscr N^{(1)} - \mathscr N^{(2)} \|_{L_2}.
\end{equation}

Using \eqref{alphapole}, \eqref{intDel0Ns}, and \eqref{Deltashort}, we show that
$$
\al_n^{(i)} = \frac{2}{\pi} + \frac{s_n^{(i)}}{n}, \quad \{ s_n^{(i)} \} \in l_2, \quad n \ge 1, \quad i = 1, 2,
$$
and obtain the estimate
\begin{equation} \label{CSDal}
\left( \sum_{n = N+1}^{\infty} \bigl| s_n^{(1)} - s_n^{(2)} \bigr|^2 \right)^{1/2} \le C(W) \| \mathscr C^{(1)} - \mathscr C^{(2)} \|_{\mathcal C}.
\end{equation}

Note that the estimates similar to \eqref{CSDM}, \eqref{CSDrho}, and \eqref{CSDal} have been obtained in~\cite{XB23} for the case of a local perturbation of Cauchy data. The constant $C$ in those estimates in \cite{XB23} depends on one of the Cauchy data sets (e.g., on $\mathscr C^{(1)}$). Nevertheless, following the proof of Lemma 2.1 in \cite{XB23}, one can see that the same constant $C = C(W)$ can be chosen for all $\mathscr C^{(1)}$ in the ball $\| \mathscr C^{(1)} \|_{\mathcal C} \le W$.

Combining the estimates \eqref{CSDM}, \eqref{CSDrho}, and \eqref{CSDal} together with the definition \eqref{defdN} of the distance $d_N$, we get the desired estimate \eqref{estCSD} for the case $\om^{(1)} = \om^{(2)} = 0$. Applying Lemma~\ref{lem:auxC}, we pass to the general case.
\end{proof}

Now, we are ready to prove the main theorem of this section.

\begin{proof}[Proof of Theorem~\ref{thm:uniCauchy}]
For the Cauchy data $\mathscr C = \{ \mathscr N, \mathscr N_0, \om, \om_0 \}$ of the problem $L$ with $(q, h, H) \in P_Q$, we have $\| \mathscr C \|_{\mathcal C} \le C(Q)$ according to the proof of Lemma~\ref{lem:sdQ}. Therefore, combining Theorem~\ref{thm:uniQ} and Lemma~\ref{lem:CSD} yields the claim:
$$
\| q^{(1)} - q^{(2)} \|_{L_2} + |h^{(1)} - h^{(2)}| + |H^{(1)} - H^{(2)}| \le C(Q) d_N(S^{(1)}, S^{(2)}) \le C(Q) \| \mathscr C^{(1)} - \mathscr C^{(2)} \|_{\mathcal C},
$$
for $N = N(Q)$ and any $(q^{(s)}, h^{(s)}, H^{(s)}) \in P_Q$, $s= 1, 2$.
\end{proof}

\medskip

{\bf Funding.} This work was supported by Grant 24-71-10003 of the Russian Science Foundation, https://rscf.ru/en/project/24-71-10003/.

\medskip

\noindent Natalia Pavlovna Bondarenko \\

\noindent 1. Department of Mechanics and Mathematics, Saratov State University, 
Astrakhanskaya 83, Saratov 410012, Russia, \\

\noindent 2. S.M. Nikolskii Mathematical Institute, Peoples' Friendship University of Russia (RUDN University), 
6 Miklukho-Maklaya Street, Moscow, 117198, Russia, \\

\noindent 3. Moscow Center of Fundamental and Applied Mathematics, Lomonosov Moscow State University, Moscow 119991, Russia.

\noindent e-mail: {\it bondarenkonp@sgu.ru}

\end{document}